\documentclass[12pt,twoside,reqno]{amsart}

\usepackage{amssymb}
\usepackage{xcolor}
\usepackage[final]{hyperref}
\usepackage[pdftex]{pict2e}

\hypersetup{unicode= false, colorlinks=true, linkcolor=blue,
	anchorcolor=blue, citecolor=green, filecolor=red, menucolor=blue,
	pagecolor=blue, urlcolor=blue}

\def\const{\text{\rm const}}

\def\PW{\text{\rm PW}}
\def\ONP{\text{\rm ONP}}

\def\supp{\text{\rm supp}\,}

\def\ti{\tilde}

\def\bs{\bigskip}
\def\ss{\smallskip}
\def\ms{\medskip}
\def\no{\noindent}

\def\PP{{\mathcal{P}}}

\def\PP{{\mathcal{P}}}

\def\BB{{\mathcal B}}

\def\CC{{\mathcal{C}}}

\def\FF{{\mathcal F}}
\def\SS{{\Sigma}}
\def\R{{\mathbb R}}
\def\T{{\mathbb T}}
\def\HH{{\mathcal H}}

\def\WW{{\mathcal W}}
\def\BB{{\mathcal{B}}}
\def\KK{{\mathcal{K}}}
\def\CC{{\mathcal{C}}}

\def\D{{\mathbb{D}}}
\def\C{{\mathbb{C}}}
\def\R{{\mathbb{R}}}

\def\T{{\mathbb{T}}}
\def\Z{{\mathbb{Z}}}

\def\si{{ \text{\rm  sinc} }}

\def\l{\lambda}
\def\L{\Lambda}
\def\a{\alpha}
\def\b{\beta}
\def\d{\delta}
\def\e{\varepsilon}
\def\s{\sigma}

\theoremstyle{plain}
\newtheorem{lem}{Lemma}[section]
\newtheorem{thm}{Theorem}[section]
\newtheorem{cor}{Corollary}[section]

\newtheorem{prop}{Proposition}[section]
\theoremstyle{remark}
\newtheorem{example}
{{\noindent\it Example}}[section]

\newtheorem{remark}
{{\noindent\it Remark}}[section]
\newtheorem{remarks}
{{\noindent\it Remarks}}[section]

\numberwithin{equation}{section}

\title{Etudes for the inverse spectral problem}

\begin{document}

	\author{N.~Makarov}
	\thanks{The first author is supported by
		N.S.F. Grant DMS-1500821}
	\address{California Institute of Technology\\
		Department of Mathematics\\
		Pasadena, CA 91125, USA}
	\email{makarov@its.caltech.edu}
	
	\author{A.~Poltoratski}
	\address{University of Wisconsin\\ Department of Mathematics\\ Van Vleck Hall\\
		480 Lincoln Drive\\
		Madison, WI  53706\\ USA }
	\email{poltoratski@wisc.edu}
	\thanks{The second author is supported by
		NSF Grant DMS-1954085.}
	
	\thanks{The work on Section 7 in the fall of 2021 was  supported by the Ministry of Science and Higher Education of the Russian Federation, agreement No 075-15-2021-602.}
	
\begin{abstract}
In this note we study inverse spectral problems for canonical Hamiltonian systems, which
encompass a broad class of second order differential equations on a half-line. Our goal is to extend the classical results developed in the work of Marchenko, Gelfand-Levitan, and Krein to broader classes of canonical systems and to illustrate the solution algorithms and formulas with a variety of examples. One of the main ingredients of our approach is the use of truncated
Toeplitz operators, which complement the standard toolbox of the Krein-de Branges theory of canonical systems. 
\end{abstract}

\maketitle

\section{Introduction} 

The main object of this note is a canonical Hamiltonian system of differential equations on a half-line. Such systems were studied around the middle of the last century by M. G. Krein, who was first to find multiple connections between 
spectral problems for such systems and structural problems in certain spaces of entire functions. The rich complex analytic content of Krein's theory was further developed in the work of L. de Branges. At present, the Krein-de Branges (KdB) theory became a standard tool in the area of spectral problems. The theory experiences a new peak of popularity in the last 20 years due to its connections to other areas  such as the theory of orthogonal polynomials, number theory, random matrices and
the non-linear Fourier transform (\cite{Denisov, Burnol, Lagarias1, Lagarias2,  Scatter, Benedek}).  In addition to the original book by de Branges \cite{dB} and
a chapter in the book by Dym and McKean \cite{DM}, more recent monographs
by C. Remling \cite{Rem} and R. Romanov \cite{Rom} contain the basics of
KdB-theory and further references.

A {\it regular} half-line canonical (Hamiltonian) system (CS) is the equation
\begin{equation}
	\Omega\dot X=z\HH X\qquad {\rm on}\quad [0,\infty).\label{eq001}
\end{equation}
Here the Hamiltonian $\HH=\HH(t)$ is  a  given $2\times 2$ matrix-function satisfying
$$\HH\in L^1_{\rm loc}[0,\infty),\quad \HH\ne 0\quad {\rm a.e.},\quad \HH\ge 0\quad {\rm a.e.}$$
The first relation means that the entries of $\HH$ are integrable on each finite interval. Systems satisfying this condition are called {\it regular}. 
The matrix $\Omega$ in \eqref{eq001} is the symplectic matrix $$\Omega=\begin{pmatrix} 0&1\\-1&0\end{pmatrix},$$ and $z\in\C$ is the 'spectral parameter'. The unknown function 
$X=X(t,z)$ is a two-dimensional vector-function on $[0,\infty)$.

Several important classes of second order differential equations, including
Schr\"odinger and Sturm-Liouville equations, Dirac and Zaharov-Shabat systems, string equations and
orthogonal polynomials can be rewritten as CS, see for instance
\cite{Dym, Rem1, Rem} or \cite{Rom}, 
 which makes
a study of such systems even more general. 

Spectral problems for differential operators ask to find  spectral data
for a given operator (direct problem) or recover an operator from its spectral 
data (inverse problem). Such problems stem from theoretical physics and lead to deep questions in analysis and adjacent areas of mathematics.  
Among the classical results in the area of inverse problems are existence and uniqueness theorems for Schr\"odinger equations by Borg and Marchenko \cite{Borg1, Borg2, Mar1, Mar2} and further extensions to broader classes of Dirac systems by Gelfand, Levitan, Gazymov, Krein and their collaborators \cite{GL, LS}. 

Our goal in this paper is to extend classical theory to broader classes of CS
and to provide explicit algorithms for the inverse spectral problem in those classes. In our approach spectral measures become symbols of truncated Toeplitz operators, bringing inverse spectral problems into the general framework of the Toeplitz approach to problems of the uncertainty principle in harmonic analysis recently developed in \cite{MIF1, MIF2, MP3, CBMS}.

Truncated Toeplitz operators, defined in Section \ref{secTT}, are bounded
and invertible in Paley-Wiener (PW) spaces of entire functions if and only if
their symbols are sampling measures in PW-spaces. We show that the class of CS 
whose spectral measures are PW-sampling is broader than those classes considered in classical theories and use the inverse Toeplitz operators to solve the inverse spectral problem. 

One of the key parts of the Gelfand-Levitan theory is the statement that
the Fourier transform $\hat\mu$ of a spectral measure $\mu$ of a regular Dirac system
has the form $\delta_0+\phi$, where $\delta_0$ is the Dirac point-mass at 0 and
$\phi$ is a locally summable function on $\R$. Recalling that the Fourier transform of the Lebesgue measure $m$ on $\R$ is $\delta_0$ (after a proper normalization), the theory considers systems which are close, in some sense, to
the free system whose spectral measure is $m$. The Gelfand-Levitan approach to
the inverse spectral problem then uses the invertibility of the operator
$$T: L^2([0,L])\to  L^2([0,L]),\  Tf=f+\phi\ast f,$$ to solve the problem. The invertibility
of $T$ is established via the Hilbert-Fredholm Lemma and uses the property that
$T$ is of the form $I+K$, where $K$ is compact. The compactness of the convolution operator $Kf=\phi\ast f$ follows from the local summability of $\phi$, which is an important condition in the classical theory.

Our approach replaces the Fredholm invertibility of $T$ with invertibility of
the truncated Toeplitz operator $L_\mu$, see Section \ref{secTT}. We show that 
spectral measures satisfying the  condition 
\begin{equation}\hat\mu=\d_0+\phi\label{eq00001}\end{equation}
form a subclass of sampling measures in PW-spaces, see Lemma \ref{lemGL}. 
(We will call \eqref{eq00001} Gelfand-Levitan condition since it appears in
their work \cite{GL, LS}.)
The algorithm then proceeds with the use of the inverse operator $L^{-1}_\mu$
which exists when $\mu$ is PW-sampling.

This approach connects inverse spectral problems with the study of sampling measures in PW-spaces, see \cite{OS} for results and further references. While describing sampling measures in a fixed space $\PW_a$
is a deep and difficult problem, measures which are sampling for all PW-spaces admit an elementary description, see Theorem \ref{PWmu}.

Our algorithm takes an especially compact form in the case when the spectral measure is periodic, see Sections \ref{secPer}-\ref{secOPUC}. In this case the inverse spectral problem connects with the theory of orthogonal polynomials on the unit circle,
see Theorem \ref{tONP}.

Given a CS with a PW-sampling spectral measure corresponding to
a fixed initial condition at the origin, it is natural to ask
if the spectral measures corresponding to other initial conditions
are also  PW-sampling. One of our examples (see Section \ref{secACDPW}) 
shows that generally this is not true. Theorem \ref{t9} gives necessary
and sufficient conditions for other spectral measures (measures from
the same Aleksandrov-Clark family) to be PW-sampling.
 

Our approach allows us to extend the classical results developed in the work of Marchenko, Gelfand-Levitan, and Krein to broader classes of canonical systems and to illustrate the solution algorithms and formulas with a variety of examples in the last section of the paper. 

$$$$

The paper is organized as follows:

\begin{itemize}
	\item In Section \ref{sec1} we review the basics of KdB-theory
	\item Section \ref{secSM} discusses the definition and main properties of spectral measures of CS
	\item In Section \ref{secR} we define representing measures for de Branges (dB) spaces generalizing Parseval's identity
	\item In Section \ref{secW} we discuss the Weyl transform which generalizes the Fourier transform in the settings of KdB-theory
	\item Section \ref{secClark} contains the definition and basic formulas related to Aleksandrov-Clark (AC) measures
	\item In Section \ref{secPW} we discuss PW-spaces as a particular case of dB-spaces corresponding to a free system
	\item In Section \ref{secPWs} we introduce PW-sampling measures.
	 Theorem \ref{PWmu} gives an elementary description of PW-class
	\item Section \ref{secExPWs} provides basic examples of PW-sampling measures
	\item In Section \ref{secGL} we discuss the class of measures appearing the classical Gelfand-Levitan theory. Lemma \ref{lemGL} shows that such measures form a subclass of PW-class.
	\item Section \ref{secCSPW} introduces the main class of CS considered in this paper, the systems whose spectral measures are PW-sampling. We formulate
	Theorems \ref{t0001} and \ref{t1} showing that PW-sampling measures correspond to Hamiltonians with a.e. non-vanishing determinants. 
	\item In Section \ref{secSym} we show how symmetries of the Hamiltonian
	relate to symmetries of the spectral measure
	\item The special case of Krein's string equation, which corresponds to 
	CS with diagonal Hamiltonians, is discussed in Section \ref{secDiag}
	\item In Section \ref{secTT} we define truncated Toeplitz operators, one of the main tools in our algorithms
	\item Using Toeplitz operators we are able to supply the proofs of Lemma \ref{lemGL} and Theorem \ref{t1} in Section \ref{secpf}
	\item Section \ref{secRep} shows how to obtain reproducing kernels
	of dB spaces from sinc functions using inverse Toeplitz operators.
	\item In Section \ref{sech11} we provide an algorithm for the recovery of
	the upper left element of a Hamiltonian $\HH$ from the spectral measure.
	\item To recover off-diagonal terms of $\HH$ we define conjugate reprokernels in Section \ref{secCK}
	\item Section \ref{secHT} contains Theorem \ref{t5}, which gives the formula
	for the generalized Hilbert transform, the unitary operator sending a chain of dB-spaces to the dual chain corresponding to a different boundary condition
	\item In Section \ref{secGmu} we provide an algorithm for the recovery
	of the off-diagonal terms of $\HH$
	\item In Section \ref{secFt} we give formulas for the Fourier transform 
	of the reproducing kernel at 0, which is used in our algorithms 
	\item Section \ref{sech22} discusses algorithms of recovery of
	the lower right element of $\HH$
	\item  Section \ref{secCD} discusses the relation between AC-dual measures
	and a change of initial condition in a given system
	\item In Section \ref{secACDPW} we give an example of a PW-sampling measure
	whose AC-dual is not PW-sampling. Theorem \ref{t9} gives a necessary
	and sufficient condition for a PW-sampling measure to have only PW-sampling measures as duals
	\item Section \ref{secPer} starts the discussion of the periodic case
	\item Section \ref{secMom} introduces a matrix of trigonometric moments for
	a periodic measure
	\item Section \ref{secHmu} contains a simplified algorithm for the recovery of the upper left element of $\HH$ in the periodic case
	\item Section \ref{secGmu1} gives a similar algorithm for the recovery of 
	off-diagonal terms
	\item In Section \ref{secOPUC} we discuss connections of the periodic case with the theory of orthogonal polynomials
	\item In Sections \ref{sec0001}-\ref{sec00011} we provide a variety of examples
	illustrating our methods and formulas
	
\end{itemize}

\section{Canonical systems and spectral measures}

\subsection{Transfer matrices, Hermite-Biehler functions and de Branges spaces}
\label{sec1}

Instead of a two-dimensional vector function $X$ one may look for a $2\times 2$ matrix-valued solution
$M=M(t,z)$ solving \eqref{eq001}. Such a matrix valued function satisfying the initial condition  
$M(0,z)=I$ is called the \textit{transfer matrix} or the \textit{matrizant} of the system. The columns of the transfer matrix $M$ are the solutions for the system \eqref{eq001} satisfying the initial conditions $\begin{pmatrix}1\\0\end{pmatrix}$ (Neumann) and $\begin{pmatrix}0\\1\end{pmatrix}$ (Dirichlet) at 0.
As a general rule we denote
\begin{equation}M=\begin{pmatrix} A& B\\C&D\end{pmatrix}.\label{eqTM}\end{equation}

An entire function $F(z)$ belongs to the Hermite-Biehler (HB) class if 
$$|F(z)|>|F(\bar z)|\textrm{ for all }z\in \C_+.$$
We say that an entire function is real if it is real on $\R$.

 For each fixed $t$, the entries of the transfer matrix $M$ of the system \eqref{eq001}, $A(z)=A(t,z)\equiv A_t(z)$, $B(z), C(z)$ and $D(z)$ are real entire functions. The functions
$$E:= A-iC,\qquad \tilde E:= B-iD$$
belong to the Hermite-Biehler class, see for instance \cite{Rem}.

For an entire function $G$ we denote by $G^\#$ its Schwartz reflection with respect to $\R$,
$G^\#(z)=\bar G(\bar z)$. We denote by $H^2$ the standard Hardy space in the upper half-plane.

For every Hermite-Biehler function $F$ one can consider the de Branges  (dB) space $\BB(F)$, a Hilbert space of
entire functions defined as
$$\BB(F)=\left\{G\ | G\text{ is entire, }\frac GF,\ \frac {G^\#}F\in H^2\right\}.$$
The Hilbert space structure in $\BB(F)$ is inherited from $H^2$:
$$<G,H>_{\BB(F)}=\left<\frac GF,\frac HF\right>_{H^2}=\int_{-\infty}^{\infty}G(t)\bar H(t)\frac{dt}{|F(t)|^2}.$$

It follows that $$\BB(F)\stackrel{\rm iso}{\subset}  L^2(1/|F|^2,\R).$$

One of the important properties of dB-spaces is that they admit an equivalent axiomatic definition.

\begin{thm}{\cite{dB}} \label{dBAxioms}
	Suppose that $H$ is a Hilbert space of entire functions which satisfies 
	
	A1) $F \in H$, $F(\lambda) = 0$ $\Rightarrow$ $F(z) \frac{z - \overline{\lambda}}{z - \lambda} \in H$ with the same norm,
	
	A2)	 $\forall \lambda \notin \mathbb{R}$, the point evaluation $F\mapsto F(\l)$ is a bounded linear functional on $H$,
	
	A3)	 $F \mapsto F^\sharp$ is an isometry in $H$.

	Then $H = \BB(E)$ for some $E \in HB$.
\end{thm}

It is not difficult to show that in every dB-space $\BB(E)$, (A2) actually holds for all $\l\in \C$. Thus, as follows from  the Riesz representation theorem, for each $\l\in\C$ there exists  $K(\l,\cdot)\in \BB(H)$ such that for any $F\in \BB(E)$,
$$F(\l)=<F,K(\l,\cdot)>_{\BB(E)}.$$
The function $K(\l,z)$ is called the reproducing kernel  (reprokernel) for the point $\l$. In the case of the dB-space $\BB(E)$, $K(\l,z)$  has the formula
$$K(\l,z)=\frac{1}{2\pi i}\frac{  E(z) E^\#(\bar \l)- E^\#(z)E(\bar \l)}{\bar \l -z }=\frac{1}{\pi }\frac{A(z) C(\bar \l) - C(z)A(\bar \l))}{\bar \l-z},
$$
where $A=(E+E^\#)/2$ and $C=(E^\#-E)/2i$ are real entire functions such that $E=A-iC$. 

The functions $E,\ \ti E$ corresponding to a canonical system \eqref{eq001} give rise to the family
of dB-spaces 
$$\BB_t=\BB(E(t,\cdot)),\qquad \tilde\BB_t=\BB(\tilde E(t,\cdot)).$$

A value $t$ is $\HH$-regular if it does not belong to an open interval on
which $\HH$ is a constant matrix of rank one. The spaces $\BB_t,\ti \BB_t$
form {\it chains}, i.e., $\BB_s\subsetneq \BB_t$ for $s<t$ and the inclusion
is isometric for regular $t$ and $s$.

 We denote by $K_t(z,w)$ and $\tilde K_t(z,w)$ their reprokernels; sometimes, when $t$ is fixed, we write $K_w(z)$ and 
 $\ti K_w(z)$.

We denote by $\Pi$ the Poisson measure on $\R$, $d\Pi(x)=\frac {dx}{1+x^2}$.
We will denote by $m$ the Lebesgue measure on $\R$. 

An entire function $F$ belongs to the Cartwright class $\CC_a$
if it has exponential type at most $a$ and
$\log |F|\in L^1(\Pi)$.

A dB-space $\BB(E)$ is called {\it regular} if
for any $F\in \BB(E)$ and any  $w\in\C$, $(F(z)-F(w))/(z-w)\in \BB(E)$. 
Any regular space $\BB(E)$ is a subspace of $\CC_a$ for some $a>0$, see
\cite{DM, Rem}.

Regular systems \eqref{eq001} give rise to chains
of regular dB-spaces $\BB_t,\ \ti \BB_t$.

 \begin{remark} 
A more general class of $2\times 2$ systems are self-adjoint systems (SAS) defined as
$$\Omega\dot X=z\HH X +QX\qquad {\rm on}\quad [0,\infty).
$$
Here $\Omega$ and $\HH$ are as before and $Q$ is another given $2\times 2$ locally summable symmetric real matrix function. 
An SAS is a general form of a second order symmetric real system.
It is well known that every regular system of this form is
equivalent to  a regular CS
 in the sense that they have the same dB spaces $\BB_t$ and therefore same spectral measures, as defined in the next section. An SAS can be reduced to a regular CS via a substitution, see \cite{Rem, Rom}.

One of the important classes of SAS is the class of Dirac systems for which $\HH=I$. Such systems play an important role in applications and constitute one of the main objects in the study of spectral problems, see for instance \cite{LS}. The class of canonical systems considered in this note includes the systems corresponding to regular Dirac systems.

$\triangle$
\end{remark}

\bs \subsection{Spectral measures}\label{secSM} There are several ways to introduce spectral measures of canonical systems. We'll make a simplifying assumption that the system has no "jump intervals", i.e., intervals on which the Hamiltonian is rank one and constant. In this case all $t\in [0,\infty)$ are
$\HH$-regular and all inclusions $\BB_s\subset\BB_t, \ti\BB_s\subset\ti\BB_t$
are isometric.
We can make this assumption because we will be mostly concerned with the case $\det\HH\neq 0$ a.e.

We informally identify absolutely continuous measures with functions
writing $\mu=f$ for the measure satisfying $d\mu(x)=f(x)dx$. In particular,
instead of writing
$\mu=cm$ for constant multiples of the Lebesgue measure we  
 write $\mu=c$.

A  measure $\mu$ on $\R$ is called Poisson-finite ($\Pi$-finite) if  $$\int\frac{d|\mu|(x)}{1+x^2}<\infty.$$
A measure on $\hat R=\R\cup \{\infty\}$ is $\Pi$-finite if it is a sum of a $\Pi$-finite measure on $\R$ and a finite point mass at infinity.

 By definition, a positive measure $\mu$ on $\R$ is a spectral measure of the CS \eqref{eq001} with the initial condition $\begin{pmatrix}1\\0\end{pmatrix}$ at $t=0$ if
$$\forall t,\qquad \BB_t\stackrel{\rm iso}{\subset} L^2(\mu).$$
(The definition is slightly more complicated in presence of jump intervals for the Hamiltonian, which we do not allow in this paper.)
It is well-known that spectral measures of regular CS are $\Pi$-finite, see for instance \cite{Rem}. In a similar way, using $\ti B_t$, one can define a spectral measure
$\ti \mu$ for the initial condition  $\begin{pmatrix}0\\1\end{pmatrix}$. 

Conversely, one of the main results of the Krein-de Branges theory says that every positive $\Pi$-finite
measure is a spectral measure of a regular CS. In general, the corresponding system may not satisfy $\det \HH\neq 0$ a.e., the restriction we are assuming in this article. Also,
HB functions corresponding to the systems considered in this paper have no zeros on the
real line. We will assume this restriction in our general
discussions of dB-spaces. (If $E$ vanishes at some point of $\R$, then all functions in $\BB(E)$ must vanish at the same point, as follows from the definition. PW-type spaces discussed in this note clearly do not have such a property.)

Every regular canonical system has a spectral measure; in fact for $\mu$ we can take any limit point of the family of measures $|E_t|^{-2}$ as $t\to \infty$, \cite{dB}. The spectral measure may or may not be unique. It is unique iff
\begin{equation}\int{\rm trace}\ \HH(t)dt=\infty.\label{eq002}\end{equation}
The case when the spectral measure is unique is called the limit point case and the case when
it is not, the limit circle case.

Finally, let us mention that spectral measures are invariant with respect to  'time' parametrizations, i.e., a  change of variable $t$ in the initial system \eqref{eq001} via an increasing homeomorphism $t\mapsto s(t)$ does not
change the spectral measure.

\subsection{De Branges measures}\label{secR}
If $E$ is an HB-function non-vanishing on $\R$, then one
can define a continuous branch of the argument of 
$E$, $\phi_E(x)=-\arg E(x)$. The function $\phi_E$ is called the phase function for the space $\BB(E)$. It is not difficult to show that $\phi_E$ is a growing
real analytic function, $\phi'>0$ on $\R$.

We call a sequence of real points discrete if it has no finite accumulation
points. A measure on $\hat\R=\R\cap\{\infty\}$ is discrete if its support on $\R$ is a discrete sequence.

For a dB-space $\BB(E)$ there exists a standard one-parameter family of discrete measures on $\R$ 
such that $\BB(E)\overset{\rm iso}= L^2(\mu)$. 
Let
$\a\in \C,\ |\a|=1$ and denote by
$t_n$ the sequence of points on $\R$ such that
$$\frac{E^\#(t_n)}{E(t_n)}=\alpha.$$
In terms of the phase function, $\phi(t_n)=\arg\alpha\ (\rm{mod }\ \pi)$.
Consider the measure
\begin{equation}\mu_\alpha=\sum_n \frac \pi{\phi'(t_n)|E(t_n)|^2}\d_{t_n}.\label{eqRM}\end{equation}
Then $\BB(E)\overset{\rm iso}= L^2(\mu_\alpha)$, for all
values of $\alpha$ except possibly one, see for instance \cite{dB}. It is well known that in the case considered
in this paper, when $E=E_t$ is an HB function corresponding
to a regular det-normalized CS, such an exceptional
value of $\a$ does not exist. 


If $E_t$ is the family of HB-functions corresponding to a system \eqref{eq001} with a fixed initial condition at 0, and $E=E_T$ for some fixed $T>0$, then the family of measures $\mu_\alpha$  constructed as above is the family of spectral
measures of the system restricted to the finite interval 
$(0,T)$ with a corresponding self-adjoint boundary condition at $0$ and various self-adjoint boundary conditions at $T$.

The Cauchy integrals of the measures $\mu_\a$ can be expressed
through the elements of the transfer matrix $M$. The following
formula will be useful to us in Section \ref{secCD}.

\begin{lem}\label{lemBA} Let $M$ be a transfer matrix
	\eqref{eqTM} of a CS \eqref{eq001}. Let $E=A-iC$
	be the HB function corresponding to the first column
	of $M$ and let the representing measures $\mu_\a$ for $\BB(E)$ be defined as in \eqref{eqRM}.
	Suppose that 
	$\BB(E)\overset{\rm iso}= L^2(\mu_{-1})$. Then
$$\KK\mu_{-1}(z):=\frac1{\pi }\int\left[\frac1{\zeta-z}-\frac \zeta{1+\zeta^2}\right]~d\mu_{-1}(\zeta)=-\frac BA+\const.$$
\end{lem}

Note that, in particular, the measure $\mu_{-1}$ is supported
at the zeros of $A(z)$, which can also be seen straight from
its definition.

\begin{proof}
The inequality $K_z(z)=||K_z||^2>0$
for the reproducing kernel $K_z$ of the space $\BB(A+iB)$
implies that $-\frac BA$ has positive imaginary part in $\C_+$. By the Herglotz representation theorem, 
$$-\frac BA =\KK \nu+\const$$
for some positive measure $\nu$ on $\hat R$.  Since $-B/A$ is analytic outside of the zeros of $A$, $\nu$ is a discrete measure whose support on $\R$ is at $\{A=0\}=\{t_n\}$, the same as $\mu_{-1}$. Since
$\BB(E)\overset{\rm iso}= L^2(\mu_{-1})$, 
the reprokernels $K_{t_n}$ form an orthogonal basis in 
$\BB(E)$, which implies that
$\nu$ cannot have a pointmass
at infinity, see for instance the proof of Theorem 22 in
\cite{dB}.
Examining the behavior of $-B/A$ near $t_n$ we 
find that $\nu$ has pointmasses of the size $-\pi B(t_n)/A'(t_n)$
at $t_n$. Since 
$$\phi=-\arg E=-\arctan\left( -\frac CA\right)\mod \pi,$$
 we have
$$\phi'=\frac 1{1+(C/A)^2}\frac {C'A-A'C}{A^2}=\frac 1{|E|^2}(C'A-A'C).$$
Since $A(t_n)=0$, 
$$\phi'(t_n)=-\frac{A'C}{|E|^2}(t_n)$$
and 
$$\frac \pi{\phi'(t_n)|E(t_n)|^2}=-\frac \pi{A'(t_n)C(t_n)}.$$
Since $\det M\equiv 1$, $C(t_n)=1/B(t_n)$ and
$$\nu(t_n)= -\pi\frac{B(t_n)}{A'(t_n)}= -\frac \pi{A'(t_n)C(t_n)}=\frac \pi{\phi'(t_n)|E(t_n)|^2}=\mu_{-1}(t_n).$$
\end{proof}

\subsection{The Weyl transform}\label{secW}
On every finite interval $[0,t]$ we define the Hilbert space $L^2(\HH, [0,t])$ associated with the Hamiltonian $\HH\geq 0$. The space consists of of 2-dim vector functions on $[0,t]$  
with the inner product 
$$<f,g>_{L^2_\HH[0,t]}=\int_0^t <Hf, g> dt.$$

 To give an alternative definition of a spectral measure for a CS we can consider the Weyl transform
$$\WW:~\begin{pmatrix} f_1\\f_2\end{pmatrix}\mapsto F(z)=\frac1{\sqrt{\pi}}\int_0^\infty\left<\HH(t)\begin{pmatrix} f_1\\f_2\end{pmatrix}, \begin{pmatrix}A_t(\bar z)\\C_t(\bar z)\end{pmatrix}\right>~dt,$$
first defined on test functions, and call $\mu$ a spectral measure if $\WW$ extends to an isometry $L^2(\HH, [0,t])\to L^2(\mu)$ for every $t\geq 0$.
 Note that then 
 \begin{equation}\BB_t=\WW L^2(\HH, [0,t])\label{eqW}\end{equation}
 and the reprokernels $K_\l(z)$ of $\BB_t$
 are obtained as Weyl transforms of the solutions, 
 \begin{equation}K_\l=\WW X_\l.\label{eq007}\end{equation}
 
 The spectral measure $\ti \mu$ can be similarly defined using the functions $B$ and $D$
 in place of $A$ and $C$ in the above formula. Spectral measures corresponding to other self-adjoint boundary conditions can also be defined using proper modifications of either one of our two definitions.

 \bs\subsection{Aleksandrov-Clark and dual measures}\label{secClark}
 Let $\varphi$ be a function from the unit ball of $H^\infty(\C_+)$
 (such functions are often called Schur functions).  Associated with every such $\varphi$
  is the family of $\Pi$-finite positive  measures $\{\s^\varphi_\a\}_{\a\in\T}$ on $\hat\R$ defined via the formula
  \begin{equation} \label{eqClark} \Re
 	\frac{\a+\varphi(z)}{\a-\varphi(z)}=py+\frac{1}{\pi}\int{\frac{yd\s^\varphi_\a
 			(t)}{(x-t)^2+y^2}}, \hspace{1cm} z=x+iy,\end{equation}
 where the number $\pi p $ can be interpreted as the pointmass of $\s^\varphi_\a$ at $\infty$. Such measures $\s^\varphi_a$ are called Aleksandrov-Clark (AC-) measures
 for $\varphi$. Families of AC-measures appear in complex function theory and applications to perturbation theory and spectral problems, see for instance \cite{PS}.
 
 At most one of the measures $\{\s^\varphi_\a\}_{\a\in\T}$ has a pointmass 
 at infinity:
  $\sigma_\alpha(\infty)\ne 0$ iff $\varphi(\infty)=\alpha$, in the sense of non-tangential limit, and $\varphi-\a\in L^2(\R)$. 
  
  It follows from the definition that all of the measures $\s^\varphi_\a$ are singular iff one of them is singular
  iff $\varphi$ is inner. All measures are discrete iff one of them is discrete iff $\varphi$ is a meromorphic inner function (MIF), an inner function which can be extended meromorphically to the whole plane.
  
   If $\varphi=\theta$ is inner, then
  	\begin{equation}\sigma_\alpha\cdot 1_\R=2\pi\sum_{\theta(\xi)=\alpha}\frac{\delta_\xi}{|\theta'(\xi)|},
  		\label{eqClarkpp}\end{equation}
 as can be deduced from the definition. 
 
 Every HB-function $E(z)$ gives
 rise to a MIF, $\theta_E=E^\#/E$. In this case the last formula is closely
 related to \eqref{eqRM} and the measures involved in the two formulas
 satisfy
 $$\s_a=|E|^2\mu_\a.$$


 We will call two $\Pi$-finite positive measures $\mu$ and $\ti\mu$ on $\R$  (AC-) dual if there exists a Schur function $\varphi$ such that  $\mu=\s^\varphi_{-1}$ and $\ti\mu=\s^\varphi_1$. We discuss
 dual measures in Section \ref{secCD} in more detail.
 
 Connection between duality of measures, as defined above, to spectral problems is provided by the property that if $\mu$ is a spectral measure
 of a CS with the Neumann initial condition at $0$, then   the spectral measure $\ti\mu$ of the same system with the Dirichlet initial condition is dual to $\mu$. Similarly, 
 measures satisfying \eqref{eqClark} with other $\a$ correspond to other initial conditions for the same system. Note also that any two measures
 $\s_\a$ and $\s_\b$, $\a\neq\b$, from the same AC-family are dual up to constant
 multiples, i. e., there exist $C_1,C_2>0$ such that $C_1\s_\a$ and $C_2\s_\b$ are dual.

 \bs\subsection{Paley-Wiener spaces}\label{secPW} Let us discuss the 'free' case $\HH\equiv I$.

 This case illustrates the connection between the Weyl transform and the  Fourier transform. We will use the following
 definition for the Fourier transform in $L^2(\R)$:
 $$ (\FF f)(\xi)\equiv \hat f(\xi)=\frac1{\sqrt{2\pi}}\int e^{-i\xi x}f(x)dx,$$
 first defined on test functions and then extended to a unitary operator
 $\FF L^2(\R)\to L^2(\R)$ via Parseval's theorem. The Paley-Wiener space $\PW_t$ of entire functions is defined as the image
 $${\rm PW}_t=\FF L^2[-t,t].$$
 By the Paley-Wiener theorem, $\PW_t$ can be equivalently defined as
 the space of entire functions of exponential type at most $t$ which belong
 to $L^2(\R)$. The Hilbert structure in $\PW_t$ is inherited from $L^2(\R)$. 
 
 Verifying the axioms in Theorem \ref{dBAxioms}, one can show that any $\PW_a$ is a dB-space. Alternatively, one can 
 arrive at the same conclusion using the first definition of a dB-space with $E(z)=e^{-iaz}\in$HB.
 
  For the transfer matrix of the free system $\HH=I$ we have
 $$\begin{pmatrix}A&B\\C&D\end{pmatrix}=\begin{pmatrix}\cos tz&-\sin tz\\\sin tz&\cos tz\end{pmatrix}.$$
 Using the standard notation $S(z)=e^{iz}$,
 $$E_t(z)=S^{-t}(z)\equiv e^{-itz},$$
 and the spectral measure of $\HH\equiv I$ is $\mu=m$, the Lebesgue measure on $\R$. The spaces $\BB_t$ are PW-spaces,
 $${\rm PW}_t=\BB(S^{-t})$$
 The reproducing kernels of ${\rm PW}_t$ are (constant multiples of) sinc functions
 $$ K_t(z,w)=\frac1\pi\si_t(z-\bar w)=\frac1\pi\frac{\sin t(z-\bar w)}{z-\bar w}.$$
 We will denote the kernel at $w=0$ by
 $$\mathring {k_t}(z):= \frac1\pi\frac{\sin tz}{z}.$$

  The Weyl transform in the free case $\HH\equiv I$ becomes the classical Fourier transform:
  $$\WW:~\begin{pmatrix} f_1\\f_2\end{pmatrix}\mapsto F(z)=\frac1{\sqrt{\pi}}\int_0^\infty\left<\HH(t)\begin{pmatrix} f_1\\f_2\end{pmatrix}, \begin{pmatrix}A_t(\bar z)\\C_t(\bar z)\end{pmatrix}\right>~dt=$$
  $$\frac1{\sqrt{\pi}}\int_0^\infty\left<\begin{pmatrix} f_1\\f_2\end{pmatrix}, \begin{pmatrix}\cos t\bar z\\ \sin t\bar z\end{pmatrix}\right>~dt=
   \frac1{\sqrt{\pi}}\int_0^\infty (f_1\cos tz +f_2\sin tz)dt = $$$$
\frac1{\sqrt{2\pi}}\int_{-\infty}^\infty f(t)e^{-itz}dt =\hat f(z),
 $$
where the function $f$ is defined on $\R$ as
$$f(t)=\begin{cases} \frac 1{\sqrt 2}(f_1(t)+if_2(t)) &\textrm{ for } t>0 \\
	\frac 1{\sqrt 2}(f_1(t)-if_2(t))  &\textrm{ for } t<0
	\end{cases}.
$$

It follows that the Lebesgue measure $m$ on $\R$ is the spectral measure for the free system. The measure is unique because $trace \ \HH\equiv 2$ and \eqref{eq002} holds.

 \bs\section{PW-sampling measures and Hamiltonians of PW-type}

 \bs\subsection{PW-sampling measures}\label{secPWs} By definition, a positive measure $\mu$ on $\R$ is  PW-sampling  ($\mu\in\PW$) if it is sampling for all Paley-Wiener spaces PW$_t$:
 $$\forall t\quad\exists C>0,\qquad \forall f\in {\rm PW}_t, \qquad C^{-1}\|f\|\le \|f\|_{L^2(\mu)}\le C\|f\|.$$
 
 Note that any PW-sampling measure must satisfy 
 \begin{equation}\sup_{x\in\R}\mu((x,x+1))<\infty, \label{eq77}
 	\end{equation}
 because otherwise
 $$\sup _{s\in\R}||f(\cdot - s)||_{L^2(\mu)}=\infty$$
 for every non-zero $f\in \PW_a$.
 
 Indeed, first notice that for any such $f$ there exists an interval
 $(x,x+\d),\ \d >0,$ on which $|f|>\e>0$. If \eqref{eq77} is not satisfied then there exists a sequence of intervals $I_n=(x_n,x_n+\d)$ such that $\mu(I_n)\to\infty$. But then
 $$||f(\cdot - (x_n+\d/2))||^2_{L^2(\mu)}\geq \e^2\mu(I_n)\to\infty,$$
 which shows that the norms are not equivalent because 
 $$||f(\cdot - (x_n+\d/2))||_{PW_a}=\const.$$
 
  It follows from \eqref{eq77} that a PW-sampling measure is  $\Pi$-finite. If $\mu=h\cdot m$ with $h^{\pm1}\in L^\infty$, then $\mu$ is obviously a PW-sampling measure. More examples will be given at the end of this section.
  
  While the description of sampling measures for a fixed $PW_a$-space is
  a difficult problem (see for instance \cite{OS}), PW-sampling measures, as defined
  above, admit a complete metric characterization.

  Given $\mu$ and $\delta>0$ we say that an interval $l\subset\R$ is $\delta$-{\it massive} with respect to $\mu$ if
 $$\mu(l)\ge \delta\textrm{ and } |l|\ge\delta.$$
 The $\delta$-{\it capacity}   of an interval $I\subset\R$ with respect to $\mu$, denoted by $C_\delta(I)$, is the maximal number of disjoint $\delta$-massive intervals intersecting $I$.

 \ms

 \begin{thm}\label{PWmu} $\mu$ is  PW-sampling if and only if

  (i) For any $x\in \R$, $\mu(x,x+1)\le \const$;

  (ii) For any $t>0$ there exist $L$ and $\d$ such that for all $I,\ |I|\ge L$,\newline
 $$C_\delta(I)\ge t|I|.$$
\end{thm}

To prove Theorem \ref{PWmu} we need to recall the following well-known results. 

A sequence $\L=\{\l_n\}\subset\C$ is called separated
if $|\l_m-\l_n|>c>0$ for a fixed $c$ and all $m\neq n$. A separated sequence
is sampling for $PW_a$  if 
$$||f(\l_n)||_{l^2}\asymp ||f||_{PW_a}$$
for $f\in PW_a$. Sampling measures on $\R$ and real sampling sequences are closely related.

\begin{prop}\label{prop1}{\cite{OS}} 
	A positive measure $\mu$ on $\R$ is sampling for $\PW_a$ iff 

1) $\mu((x,x+1))<C<\infty$ for some $C$ and all $x\in \R$ and

2) There  exists a sampling  sequence $\{\l_n\}$ for $\PW_a$ and $\d>0$ such that the  points $\l_n$ have disjoint 
neighborhoods $U(\l_n)$ satisfying $\mu(U_n)>\d$.
\end{prop}

For $\L\subset\R$ define the lower Beurling uniform density of $\L$ as
$$D_-(\L)=\lim_{L\to\infty}\min_{x\in\R}\frac{\#(\L\cap (x,x+L))}L.$$

\begin{prop}\label{prop2}
A real separated sequence $\L$ is sampling for $PW_a$ if $\pi D_-(\L)>a$ and 
not sampling if $\pi D_-(\L)<a$. 
\end{prop}

This statement is essentially contained in the work of Beurling \cite{Beurling}; see also \cite{Ja} for a different version.

\begin{proof}[Proof of Theorem \ref{PWmu}]
	Let $\mu$ satisfy the conditions of the theorem and let $t$ be a positive number. Then for $\d=\d(t)$ from (ii) one can choose
	a sequence of disjoint $\d$-massive intervals whose centers $\l_n$ form  a sequence $\L$ satisfying $D_-(\L)>t/2$.
 Hence  it is a sampling sequence for $PW_{t/2}$ by Proposition \ref{prop2}. Together with
	(i), Proposition \ref{prop1} implies that $\mu$ is a sampling measure for $PW_{t/2}$.
	
	In the opposite direction, if (i) does not hold, then $\mu$ is not a sampling measure for any $PW_a$ as was discussed above. Suppose that  (i) holds but $\d$ from (ii) does not exist for some $t>0$. If $\mu$ is sampling for $\PW_{t+\e}$, then by 
	Proposition \ref{prop1} one can choose $U(\l_n)$ such that  $\L=\{\l_n\}$ is sampling for $\PW_{t+\e}$. Then $D(\L)>t$ by Proposition \ref{prop2}. Since the sequence $\L$ is separated, one can choose
	$U(\l_n)$ to be $\d$-massive and obtain a contradiction.
	Therefore, $\mu$ is not sampling for $PW_{t+\e}$. 
	\end{proof}

For a PW-sampling measure  $\mu$  we denote by PW$_a(\mu)$ the Hilbert  spaces of entire functions which have the same elements as the spaces PW$_a$ but the $L^2(\mu)$ scalar product. From the axiomatic characterization of dB-spaces, it follows that PW$_a(\mu)$ are dB-spaces. By construction, all spaces PW$_a(\mu)$ are embedded isometrically in $L^2(\mu)$.

\subsection{Basic examples of PW-sampling measures}\label{secExPWs}

 \begin{enumerate}
 
 \item As was mentioned before, any measure $\mu$ of the form 
 $$d\mu=w(x)dx,\ w(x)\asymp 1$$
  is  PW-sampling.
 The upper bound $w<C$ can be replaced with  \eqref{eq77}. This can be seen directly
 from the estimate 
 $$||f'||_\infty \leq a\sqrt{2a}||f||_{PW_a}$$
  which holds for any $f\in \PW_a$.
  
  \bs

\item Any measure that decays to 0 at infinity, i.e., any
measure satisfying
$$\mu((x,x+1))\to 0 \ {\rm as}\ x\to\infty,$$
is not  PW-sampling. To see this, note that shifts of the sinc function
have the same PW-norm but
$$\left|\left|\frac{\sin (z-N)}{z-N}\right|\right|^2_{L^2(\mu)}\lesssim \sum_{n\in \Z} \frac{\mu((n,n+1))}{(N-n)^2+1}\to 0\ \text{ as } \ N\to\infty.$$
In particular, any finite measure is not  PW-sampling.
  
\bs

\item
Let $\mu= m +\sum_0^\infty \a_n \d_{x_n},$ where $0<\a_n<C$ and $\{x_n\}\subset \R$ is a separated sequence. Then $\mu$
 is  PW-sampling. 
 
 Also, if $\mu\in \PW$ and
$\nu\geq 0$ decays at infinity, in particular if $\nu$ has compact support, then $\mu+\nu\in \PW$. 
This and further examples of that kind are obtained from the
following argument.

If $\mu$ satisfies  (ii) from Theorem \ref{PWmu} and $\nu$ is any non-negative measure, then $\mu+\nu$ satisfies (ii).
If $\mu$ and $\nu$ both satisfy (i), then so does $\mu+\nu$. Hence, if $\mu\in \PW$ and $\nu\geq 0$ satisfies (i), then
$\mu+\nu\in\PW$. 

In particular, addition of finite positive measures with compact support or locally finite positive measures decaying at infinity does not
change the PW-sampling property. 
  
\bs

\item It follows from Theorem \ref{PWmu} that a positive periodic measure is  PW-sampling  if and only if it has
locally infinite support. In particular, the measures $\mu(x)=1+a\cos x ,\ a\in[-1,1]$, considered in 
Section \ref{secEx}, are $\PW$-sampling,
but a measure which has only finitely many point masses on each period is not. 
  
\bs

\item 

Any discrete measure $\mu$ satisfying
$$D_-(\supp\mu)<\infty$$
is not  PW-sampling, as follows from  Propositions \ref{prop1} and \ref{prop2}. In particular, any measure supported on a separated sequence is not  PW-sampling. 

\item PW-sampling measures of the form $\mu=1_S\cdot m,\ S\subset \R$ are described
by the Paneah-Logvinenko-Sereda theorem \cite{Paneah, LS1}.
 A set $S\subset \R$ is relatively dense if there exists $r,\delta >0$ such that $$|(x-r,x+r)\cap S|>\delta$$ for all $x$. An equivalent reformulation of the theorem in one dimension (the case studied by Paneah) is that $\mu$ is PW-sampling iff $S$ is relatively dense. This statement also follows from Theorem \ref{PWmu}.
\end{enumerate}

%
%
%
%
%
%

\subsection{Relations to Gelfand-Levitan theory}\label{secGL}

  We will say that a positive $\Pi$-finite non-discrete measure  $\mu$ is a GL-measure ($\mu\in$GL) if it satisfies the Gelfand-Levitan-type condition 
  \begin{equation}\label{eqGL}\FF\mu=C\d_0 +f,\end{equation} for some 
  function $f\in L^1_{\rm loc}(\R)$ and $C>0$. Here the Fourier transform $\FF\mu$ is understood in the sense of distributions. The condition means that $\mu$ is, in some sense, close to a constant multiple
  of the Lebesgue measure, $\frac C{\sqrt{2\pi}}m$, whose Fourier transform is $C\d_0$.
  As was mentioned in the introduction, measures of this class are considered in the classical Gelfand-Levitan theory of inverse spectral problems. The relation between the class of PW-sampling measures and the class GL defined above is established by the following statement.
  
  \begin{lem}\label{lemGL}
  $$\textrm{GL}\subset PW.$$
  	\end{lem}
  
  The inverse inclusion clearly does not hold. The measure $1+\cos x$ considered in Section \ref{secEx} is
  obviously PW-sampling but not GL because its Fourier transform consists of three pointmasses.
  Furthermore, any periodic measure with locally infinite support is a PW-sampling measure, but not a GL-measure, unless it is
  a constant multiple of $m$. 
  For a non-periodic example, If $\nu$ is any even finite measure with a non-trivial singular part, then $f=\hat \nu$ is a real bounded function and $\mu=(C+f)m\in \PW$ for large enough $C$. But $\mu\not\in $GL
   because $\FF(\mu)=D\delta_0-\nu$ and $\nu\not\in L^1_{\rm loc}$. 
  
   The proof of Lemma \ref{lemGL} is postponed until Section \ref{secTT}.
  
  \begin{remark}
  	The measures considered in Gelfand-Levitan theory (see for instance \cite{LS}, Chapter XII) are the spectral measures of Dirac systems with regular potentials on the half-line. Such measures have non-trivial absolutely continuous parts,
  	 satisfy \eqref{eqGL} and some additional conditions. 
  	 
  	 The class GL as defined above is the class of measures for which some of the methods of the classical theory may still work, although it is substantially broader than the set of measures actually considered in \cite{LS} or \cite{Rem1}.
  	
  	As will be seen in the proof of Lemma \ref{lemGL}, the restriction
  	that a GL-measure should not be discrete can be further weakened to the condition that it should not be supported on a zero set of
  	a PW-function (a sequence of finite Beurling-Malliavin exterior density, see
  	for instance \cite{Koosis}). 
  	
  	$\triangle$
  \end{remark}


  \bs\subsection{Canonical systems of PW-type}\label{secCSPW} We say that two Hilbert spaces $H_1$ and $H_2$ are equal as sets,
  and write $H_1\doteq H_2$, if the spaces have the same elements, but not necessarily the same scalar product. We say that
  a dB-space $\BB$ has exponential type $T$ if $T$ is the smallest constant with
  the property that all functions in $\BB$ have exponential type at most $T$.
  
  A formula by Krein, see for instance \cite{Rem, Rom}, allows one to calculate the type $T$ of a dB-space $\BB_t$
  corresponding to a Hamiltonian $\HH(t)$:
  \begin{equation}T=\int_0^t\sqrt{\det \HH(t)}dt.\label{eqType}\end{equation}

  Two SAS are {\it equivalent} if the corresponding chains of
 dB-spaces $\BB^1_t$ and $\BB^2_t$ satisfy
 $$\BB^1_t=\BB^2_{s(t)}$$
 (norms and all) for some increasing homeomorphism $s:\R_+\to\R_+$.
 The systems are {\it strongly equivalent} if $s(t)=t$.
 Basic results of KdB theory imply that two CS are equivalent iff they have the same spectral measures.

The condition $\HH\neq 0$ a.e., imposed on every CS \eqref{eq001}, together with
\eqref{eqW},
 implies that for any $t_2>t_1\geq 0$, $\BB_{t_1}\subsetneq \BB(t_2)$.
 
  Let $\HH$ be a Hamiltonian of a canonical system \eqref{eq001}. We say that $\HH$ is of PW-type ($\HH\in \PW$)  if for any $t>0$ there exists $s=s(t)>0$ such that $s(t)\to\infty$ as $t\to\infty$ and
 \begin{equation}\BB_t(\HH)\doteq {\rm PW}_s.\label{eq20}\end{equation}
 We call the corresponding system a PW-type system.

 \begin{thm}\label{t0001} If $\HH$ is of PW-type, then
 \begin{equation}
 	\det \HH\ne 0\quad{\rm a.e.},\qquad \int_0^\infty\sqrt{\det\HH(t)}dt=\infty. \label{eq005}
 \end{equation}
 Also, we have
 $$\BB_t\doteq{\rm PW}_{s(t)},\qquad s(t):= \int_0^t\sqrt{\det(\HH)}.$$
 \end{thm}

\begin{proof} 
		We can assume that $\HH$ is trace-normalized.
	Let $\mu$ be a spectral measure corresponding to  $\HH$. Then $\mu\in \PW$ and
$\det \HH\neq 0$ a.e., see the proof of Theorem \ref{t1} in Section \ref{secpf}. 
The equation in \eqref{eq005} follows from $s(t)\to\infty$ and Krein's type formula, together with the remaining equations in the statement.
\end{proof}

  The change of time $t\mapsto s=s(t)$ in the above theorem allows us to transform any PW-type system, or more generally any canonical system satisfying \eqref{eq005}, into a canonical system with $$\det \HH=1\quad{\rm a.e.}$$
 We will call such systems {\it det-normalized}. A regular (locally summable) Hamiltonian will remain regular under det-normalization.

  \begin{thm}     Consider a regular CS with the Hamiltonian $\HH,\ \det\HH\ne 0$ a.e. Then there exists an equivalent regular system with the Hamiltonian $\tilde \HH$ such that $\det\tilde \HH=1$ a.e.    \end{thm}

 \begin{proof}    The function
 	$$s(t)=\int_0^t\sqrt{\det \HH}$$
 	is a $W^{1,1}_{\rm loc}$-diffeomorphism because $\sqrt{\det \HH}$  is in $L^1_{\rm loc}$ and $\det \HH> 0$ a.e.
 Consider the inverse diffeomorphism $t(s)$ and set 
 	$$\tilde H(s)=(\HH\circ t)(s)
 	\dot t(s).$$ Then $\tilde\HH\ge 0$, is locally integrable, and $ \det \tilde\HH=1$ a.e. Clearly, the systems corresponding to  $\tilde\HH$ and $ \HH$
 	are equivalent.
 	\end{proof}

As was mentioned before, a change of time variable, in particular det-normalization, does not affect the spectral measure of the system.

 The converse of the last theorem is not true: our next example shows that  there are det-normalized canonical systems which are not of PW-type,
see also \cite{BR}.

\begin{example}

	\no  Consider the Hamiltonian $$\HH=\begin{pmatrix} t^m&0\\0&t^{-m}\end{pmatrix},\qquad m\in (-1, 1),$$
	and denote $$m=2\nu-1.$$
	Let $J_\nu$ denote the Besssel function of the first kind and let
	$$J_\nu(\lambda)=\lambda^\nu F_\nu(\lambda).$$
	Then $F_\nu$ is an entire function. For the solution
	$\begin{pmatrix} A // C\end{pmatrix}$ of the CS with the Hamiltonian $\HH$
	we have:
	
	\begin{lem} $$A=F_{\nu-1}(zt),\qquad C=t^{2\nu} z F_\nu(zt).$$
	\end{lem}

The lemma can be established by direct calculations.
	
	\ms\no We will use the following asymptotic of the Bessel functions on the positive axis:
	$$J_\nu(\lambda) =\frac{\cos\left[\lambda-(2\nu+1)\frac\pi4\right]}{\sqrt\lambda}+O\left(\lambda^{-3/2}\right),\qquad \lambda\to\infty.$$
	The amplitudes of $|A|^2$ and $|C|^2$ are $t^mx^{-m}$ and $t^{-m}x^{-m}$ respectively. Hence for each fixed $t$,
	$$\frac1{|E(x)|^2}\approx       x^m,\qquad x\to\infty.$$ 
 It follows that the measure $|E_t(x)|^{-2}dx$ is not sampling for $\PW_t$
 for any $t$, except in the case $m=0$. Therefore the corresponding spaces $\BB_t$ are not equal to $\PW_t$ as sets and the system is not of PW-type.

\end{example}

 \begin{remark} $ $

 	 (a) All regular Dirac systems are of PW-type as follows from Lemma \ref{lemGL}
 	and the results of \cite{LS}, see also \cite{BR}.
 	
(b)  Characterization of PW-Hamiltonians is a difficult problem even in the diagonal case. So far we know that
$$\det\HH\ne0 \qquad \not\Rightarrow\qquad  PW$$
but 
$$\det\HH\ne0,\quad \HH\in W^{1,1}_{\rm loc}    \qquad        \Rightarrow\qquad  PW,$$
as follows from the property that regular Dirac systems correspond 
to CS with   $W^{1,1}_{\rm loc}$-Hamiltonians.

  (c) We will also see that  $$\begin{pmatrix}h_{11}&h_{12}\\h_{12}&h_{22}\end{pmatrix}\in\PW\ \not\Rightarrow \ 
 \begin{pmatrix}h_{22}&-h_{12}\\-h_{12}&h_{11}\end{pmatrix}\in\PW,$$ Example \ref{exnPW}. The second Hamiltonian corresponds to the switch between the Dirichlet and Neumann initial conditions, see Section \ref{secSym}.
 
 $\triangle$
\end{remark}

  The next statement relates PW-type systems and PW-sampling measures.

 \begin{prop} Suppose $\det(\HH)=1$ a.e., and let $\mu$ be the (unique) spectral measure of the corresponding CS. Then
 $$\mu\in{\rm(PW)}\qquad \Leftrightarrow\qquad \HH\in{\rm(PW)}.$$
 Moreover, if either holds, then
 $$\forall t,\quad \BB_t(\HH)= {\rm PW}_t(\mu).$$\end{prop}

\begin{proof}
 A $\Pi$-finite measure $\mu$ admits a unique, up to parametrization, chain of regular de Branges spaces isometrically embedded in $L^2(\mu)$, as follows from the results of \cite{dB}. Since both $\BB_t(\HH)$ and $\PW_s(\mu)$ form such chains, we must have
$$\BB_t(\HH)=\PW_s(\mu) \doteq \PW_s.$$

By Krein's formula, the space $\BB_t$ consists
of functions of exponential type at most
$$\int_0^t\sqrt{\det\HH(x)}dx,$$
which is equal to $t$ due to the det-normalization. Since $\PW_s(\mu)$ consists of functions of type $s$, we must have $\BB_t(\HH)=\PW_t(\mu)$.
\end{proof}

  By one of the main theorems of Krein-de Branges theory, every positive $\Pi$-finite measure
 on $\R$ is a spectral measure for a canonical system with a locally summable (trace-normalized) Hamiltonian \cite{dB}. We obtain the following
 corollary.

 \begin{thm}\label{t1} Let $\mu$ be a PW-sampling measure. Then there exists a det-normalized $\HH$ such that $\mu$ is its spectral measure. Two det-normalized Hamiltonians have the same  spectral measure iff there is $k\in\R$ such that
 $$\check \HH=\HH_k:=T_k^*\HH T_k, \qquad T_k=\left(
                                 \begin{array}{cc}
                                   1 & k \\
                                   0 & 1 \\
                                 \end{array}
                               \right).
 $$
 \end{thm}

We postpone the proof until Section \ref{secpf}.

\begin{remark}\label{remHk} In the above formula
$$\HH_k=\begin{pmatrix}h_{11}&h_{12}+kh_{11}\\h_{12}+kh_{11}&h_{22}+2k h_{12}+k^2h_{11}
\end{pmatrix}$$
and $$M_k:=T^{-1}_kMT_k=\begin{pmatrix}A-kC&B+kA-kD-k^2C\\C&kC+D
\end{pmatrix}$$
is the matrizant for $\HH_k$.

 $\triangle$
\end{remark}

\subsection{Symmetries}\label{secSym}

For a Hamiltonian 
$$\HH= \begin{pmatrix}h_{11}&h_{12}\\h_{12}&h_{22}\end{pmatrix}$$
we denote
$$\breve \HH= \begin{pmatrix}h_{11}&-h_{12}\\-h_{12}&h_{22}\end{pmatrix}{\rm and }\ \ti \HH=\begin{pmatrix}h_{22}&-h_{12}\\-h_{12}&h_{11}\end{pmatrix}.$$

For a measure $\mu$, $\breve \mu$ is the measure satisfying $\breve\mu (e)=\mu(-e),\ e\subset \R,$
and $\ti\mu$ is the AC-dual measure as defined in Section \ref{secClark}.

We denote by $\mu_\HH$ the spectral measure corresponding to the system with the Hamiltonian $\HH$.

\begin{lem}\label{l1} $ $
	
	(i) The involution
$\HH\mapsto\breve{\HH}$
descends to the involution $\mu\mapsto \breve\mu$ for the corresponding spectral measures: $\mu=\mu_\HH,\ \breve \mu= \mu_{\breve\HH}$.

(ii)  The involution
$H\mapsto\tilde{H}$
descends to the involution $\mu\mapsto \tilde\mu$. In other words,
$$ \mu_{\tilde H}=\tilde\mu_H$$
\end{lem}

\begin{proof} $ $
	
	(i) Let $\tau=\begin{pmatrix}1&0\\0&-1\end{pmatrix}$. Then $\tau^2=$id and $\Omega\tau=-\tau\Omega$. Note that
 $$\tau H\tau=\breve H$$
 Computations show that $N(z)=M^{(\tau)}(-z),{\rm where }\ M^{(\tau)}:=\tau M\tau,$
 is the matrizant of $\breve H$: from $\tau\Omega\dot M\tau=z\tau HM\tau$ we get
 $$-\Omega\dot M^{(\tau)}=z\tau H\tau\tau M\tau=z\breve H M^{(\tau)}.$$
Furthermore,
 $$N(x)=\begin{pmatrix}A(-z)&-B(-z)\\-C(-z)&D(-z)\end{pmatrix}$$
 and $$|E_N(x)|^2=|A(-x)+iC(-x)|^2=|E(-x)|^2$$
 It follows that $\breve \mu$ is the spectral measure of $\breve H$.

  (ii) We use the same argument but with $\tau=\Omega$. We have
 $$\Omega(\Omega \dot M\Omega)=z\Omega HM\Omega=z(-\Omega H\Omega)(\Omega M \Omega)$$
 It follows that
 $$\Omega M \Omega= \begin{pmatrix}-D&C\\B&-A\end{pmatrix}$$
 is the matrizant of $\tilde H=-\Omega H\Omega$.
 \end{proof}

 \bs\subsection{Even measures and diagonal Hamiltonians}\label{secDiag} We call a measure
 $\mu$ on $\R$ even if $\mu=\breve\mu$, i.e., $\mu(x)=\mu(-x)$.

 \begin{thm}\label{t3} If the Hamiltonian $\HH$ is diagonal, then  the spectral measure $\mu_\HH$
  is even. Conversely,  if $\mu$ is an even positive $\Pi$-finite measure, then there 
  exists a diagonal Hamiltonian $\HH$ such that $\mu=\mu_\HH$. \end{thm}

 \bs
 \begin{proof}  The first  part follows from the last lemma. Alternatively, we can argue as follows.
 	
 	Let $\HH={\rm diag}(h_1,h_2)$, then the system is
$$\dot C=zh_1A,\qquad \dot A=-zh_2 C.$$
Denote
$$\ti A(z,t)=A(-z,t),
\qquad \ti C(z,t)=-C(-z,t).$$
We have $\ti A=A$ and $\ti C=C$ because the functions satisfy the same system and the same ICs. We derive the reality condition
$$E(-z)=E^\#(z),$$
and therefore the symmetry of $\mu$.

As follows from Theorem \ref{t1} and Lemma \ref{l1}, if the measure is symmetric, then there is a number $k$ such that
$$\breve H-H=\begin{pmatrix}0&kh_{11}\\kh_{11}&\ast\end{pmatrix}$$
which implies
$$-2h_{12}=kh_{11},$$
so the ratio of the functions $h_{12}$ and $h_{11}$ is constant. Applying Theorem \ref{t1} again we see that the Hamiltonian
can be made diagonal.
\end{proof}

\begin{remark} CS with diagonal Hamiltonians are equivalent to so-called
	Krein's strings, see the book by Dym and McKean \cite{DM}. In particular,
the converse statement in Theorem \ref{t1}  also follows from Krein's theorem which says that if $\mu$ is a $\Pi$-finite even measure, then it is a spectral measure of a string. 

$\triangle$
\end{remark}

  Combining the last statement with Theorem \ref{t1},   in the PW-situation we arrive to the following conclusion.

 \begin{thm} There is a 1-to-1 correspondence between even PW-sampling measures and diagonal det-normalized Hamiltonians of PW-type. \end{thm}

  \bs\section{Solution formulae  for the inverse problem}

  \bs\subsection{Truncated Toeplitz operators}\label{secTT} Let $t>0$ and let $\mu$ be a positive measure such that PW$_t\subset  L^2(\mu)$. Then we can define a bounded operator
  $$L\equiv L_{\mu, t}: ~{\rm PW}_t\to {\rm PW}_t$$
  by the (symbolic) formula
  $$f\mapsto P_t (f \mu),$$
  where $P_t$ is the "projection" to PW$_t$. The formula makes perfect sense if $\mu\in L^\infty$ but we need to interpret it in terms of quadratic forms in the general case:
   \begin{equation}\forall f,g\in {\rm PW}_t,\qquad (Lf,g)=\int f\bar g d\mu.\label{eqTTdef}\end{equation}
   Clearly, $L$ is a well defined bounded non-negative operator. 
   
   The following statement can be deduced directly from the definition.

   \begin{lem}\label{lemL} $L_{\mu, t}$ is an  invertible (positive) operator iff $\mu$ is sampling for ${\rm PW}_t$.  \end{lem}

In regard to the family of PW-sampling measures we obtain the following

\begin{cor}
	If $\mu\in \PW$, then the operator $L_\mu$ defined by \eqref{eqTTdef} is bounded and invertible in every $\PW_t,\ t>0$.
	\end{cor}
  
  Consider now the dB spaces
  $$\BB_t={\rm PW}_t(\mu)$$
  and the  operators
  $$j=j_t:~{\rm PW_t}\to\BB_t,\qquad f\mapsto f.$$

  \begin{thm}  If $\mu\in$PW, then for any $t>0$,
  $$j_t^*j_t=L_{\mu}.$$\end{thm}
  \begin{proof} If $f,g\in{\rm PW}_t$, then
  $$(j^*j f,g)_{\rm PW}=(jf,jg)_\BB=(jf,jg)_\mu=\int f\bar g~d\mu.$$
  \end{proof}

\subsection{Postponed proofs}\label{secpf}

Using the operators $L_\mu$ we can now supply the proofs of the lemma on GL-measures and the theorem on Hamiltonians with PW spectral measures.

\begin{proof}[Proof of Lemma \ref{lemGL}]
	Let $\mu\in GL$, $\hat \mu = \d_0 +\phi,\ \phi\in L^1_{loc}$. 
	
	Let us first show that $L_\mu$ has trivial kernel in every $\PW_a$.
	If $f\in \PW_a,\ f\not\equiv 0$ is in the kernel, then the measure $f\mu$ annihilates $\PW_a$.
	In particular
	$$\int f^\# f d\mu=\int |f|^2 d\mu=0,$$
	which implies that $\mu$ is a discrete measure supported on the zero set of $f$. Since GL-measures are non-discrete,
	we obtain that $L\mu$ is injective on every $\PW_a$.
	
	Let $\mu_a$ be the measure such that $\hat \mu_a=1_{[-2a,2a]}\hat\mu$. Then $\mu_a=(1+\psi)m$ where
	$\psi=\FF^{-1}(\phi_a),\ \phi_a=1_{[-2a,2a]}\phi,$ is a bounded function. Note that the restriction of $L_\mu$ onto $\PW_a$ is equal to
	$L_{\mu_a}$. Therefore 	$L_{\mu_a}$ has trivial kernel in $\PW_a$.
	
	Note that
	$$\FF(L_{\mu_a}f)=1_{[-a,a]}(\hat f  + \hat f \ast \phi_a)= \hat f + 1_{[-a,a]}\hat f \ast \phi_a.
	$$
	
	Hence the operator 
	$$\hat L_{\mu_a}:L^2([-a,a])\to L^2([-a,a]),\ \hat L_{\mu_a} \hat f=\FF(L_{\mu_a} f),$$
	is of the Fredholm form
	$$\hat L_{\mu_a}=I +K, \ \ Kf=P_a( f \ast \phi_a),$$
	where $P_a$ is the orthogonal projection $L^2([-3a,3a])\to L^2([-a,a])$. Since $\phi_a$ is summable, and $P_a$ is bounded, 
	$K$ is compact. Since the kernel of $L_{\mu_a}$ is trivial, so is the kernel of $\hat L_{\mu_a}$. By the Fredholm alternative,
	$\hat L_{\mu_a}$ is invertible.
	
	Since $\phi_a$ is summable, $\hat L_{\mu_a}$ is bounded. 
	
	Hence $L_\mu$ is bounded and invertible in any $\PW_a$. By Lemma \ref{lemL}, $\mu\in \PW$.
\end{proof}

Let us now prove Theorem \ref{t1}.

We denote by $\PW_{[a,b]}(\mu)$
the orthogonal difference $\PW_{b}(\mu)\ominus\PW_{a}(\mu)$.
We will need the following

\begin{lem}
	Let $\mu\in\PW$. Then for any $0\leq a<b\leq c$ there exists
	a constant $d=d(c,\mu)>0$ such that for any $f\in\PW_{[a,b]}(\mu)$
	$$\int_a^b|\hat f(x)|^2 dx > d ||f||^2_2.$$
\end{lem}

\begin{proof}
	Let $f\in\PW_{[a,b]}(\mu)$. Note that then $f\in\PW_b$. 
	Denote by $f_a$ the orthogonal projection of $f$
	onto $\PW_a$ in $\PW_b$ and let $f_b=f-f_a$. Then the integral
	in the statement is $||f_b||_2^2$.
	
	Note that $f\in\PW_{[a,b]}(\mu)$ is equivalent to $L_{\mu,b}f\perp \PW_a$, which implies that $\supp \FF(L_{\mu,b}f)\cap (-a,a)=\emptyset$. 
	Since $L_{\mu,a}$ is invertible,
	$$||L_{\mu,b}f_b||_2^2=\int_{-b}^b |\FF(L_{\mu,b}f_b)(x)|^2dx\geq \int_{-a}^a |\FF(L_{\mu,b}f_b)(x)|^2dx=$$
	$$\int_{-a}^a |\FF(L_{\mu,b}(f-f_a))(x)|^2dx=
	\int_{-a}^a |\FF(L_{\mu,b}f_a)(x)|^2dx=||L_{\mu,a} f_a||_2^2
	>d_1 ||f_a||_2^2,$$
	for some $d_1=d_1(c)$.
		Since $L_{\mu,b}$ is bounded, it follows that
		$$||f_b||_2^2>d_2 ||f_a||_2^2,$$
		for some $d_2=d_2(c)$.
	\end{proof}	
	
	\begin{proof}[Proof of Theorem \ref{t1}] Since PW-sampling measures are $\Pi$-finite, $\mu$ is the spectral measure of some regular trace-normalized canonical system with Hamiltonian $\HH$. 
	We need to show that the set $\{t | \det\HH(t)=0\}$ has Lebesgue measure
	zero. If that is the case, we can renormalize the system with the change of 
	variable 
	$$s(t)=\int_0^t\sqrt{\det\HH}.$$
	
	To show that $|\{t | \det\HH(t)=0\}|=0$, recall that any $\Pi$-finite
	measure admits a unique chain of regular de Branges spaces. Since
	both $\PW_s(\mu)$ and the spaces $\BB_t$ corresponding to $\HH$ constitute
	such chains, $\PW_s=\BB_t$ for some $s=s(t)$. According to Krein's formula
	for the exponential type \eqref{eqType}, $s(t)$ is defined as above.
	
	Assume now that $|\{t | \det\HH(t)=0\}\cap (0,T)|>0$ for some $T>0$. Let $f\in L^2(\HH)$
	be a non-zero function such that  $$\supp f\subset\{t | \det\HH(t)=0\}\cap (0,T).$$ 
	It follows from the type formula that for any $\e>0$ one can choose disjoint intervals  
	$$[a_k,b_k]\subset (0,T)\setminus \supp f,\ 0\leq a_1<b_1<a_2<...<b_n\leq T$$  such that 
	$$\sum_{k=1}^n s(b_k)-s(a_k)=s(T)-\e.$$
	Put $F=\WW f$. 
	Then $F\in \PW_{s(T)}$ and $\supp\hat F\subset [0,s(T)]$. According to the last lemma,
	$$\int_0^{s(a_1)} |\hat F(x)|^2dx+\int_{s(b_1)}^{s(a_2)}|\hat F(x)|^2dx+...+
	\int_{s(b_n)}^{s(T)}|\hat F(x)|^2dx>d||F||_2^2.$$
	Recall that the Weyl transform $\WW$ is unitary and therefore $||F||_2^2>0$. Since 
	$$\e=\left|\ [0,s(T)]\setminus \cup [s(a_k),s(b_k)]\ \right|$$
	 can be chosen arbitrarily small, we obtain a contradiction.

	For the second part of the statement, notice that if $\HH$ and $\check\HH$
	have the same $\PW$-sampling spectral measure, then the det-normalization for
	both Hamiltonians implies $\BB_t=\check \BB_t=\PW_t(\mu)$ for all $t$.
	Then for the HB functions $E_t=A_t-iC_t$ and $\check A_t - i\check C_t$  we have
	
\begin{equation}
\begin{pmatrix}
	\check A_t\\ \check C_t
\end{pmatrix}
=R^{-1} \begin{pmatrix}
	A_t\\C_t
\end{pmatrix}
\label{eqmat}\end{equation}
fore some constant matrix $R\in {\rm SL}(2,\R)$, see for instance \cite{Rem}, 
Theorem 2.2. 

Note that if $\begin{pmatrix}
	A_t\\C_t
\end{pmatrix}$ is a solution to a CS with a Hamiltonian $\HH$ then 
$R^{-1}\begin{pmatrix}
	A_t\\C_t
\end{pmatrix}  $
is a solution to a system with the Hamiltonian $R^*\HH R$. 
Here we use the property that any $SL(2,\R)$ matrix is symplectic, i.e., $R^*\Omega R=\Omega$.

From the uniqueness of systems on $(0,t)$ with the same HB-function corresponding to $t$ (see for instance Theorem 13 in \cite{Rom}) it follows that all $\begin{pmatrix}
\check A_s\\ \check C_s
\end{pmatrix},\ s<t$ satisfy \eqref{eqmat} with $t=s$ and the same constant matrix $R$. Choosing larger $t$ we conclude that \eqref{eqmat}
must hold with the same $R$ for all $t$. Finally, the property that 
$$\begin{pmatrix}
	A_t\\C_t
\end{pmatrix}(0)=
\begin{pmatrix}
	\check A_t\\ \check C_t
\end{pmatrix}(0)=
\begin{pmatrix}
	1 \\ 0
\end{pmatrix}$$
together with $R\in {\rm SL}(2,\R)$ imply that $R$ is upper triangular with $1$ on the diagonal.
	\end{proof}

  \bs\subsection{Reprokernels}\label{secRep} Let $\mu$ be a sampling measure for $\PW_t$ for some fixed $t>0$.
  Consider the de Branges space $\BB=\PW_t(\mu)$, $\BB\doteq{\rm PW}_t$. Then the truncated Toeplitz operator $L=L_{\mu,t}$,  $L:{\rm PW}_t\to{\rm PW}_t$ defined as in the last
  section is bounded and invertible. For $w\in\C$ we denote by $K_w$ and $\mathring K_w$ the reprokernels in $\BB$ and PW$_t$ respectively.

\begin{thm}\label{t2} For all $w\in \C$, 
$$j\left[L^{-1}\mathring K_w\right]=K_w.$$\end{thm}

\begin{proof} For any $f$ and $g$ in PW$_t$ we have
$$ (jf,jg)_\BB=(f,Lg)_{\rm PW}.$$
For $g=j^{-1} K_w$ we obtain
$$(jf, K_w)_\BB=(f, Lj^{-1}K_w)_{\rm PW}.$$
The left-hand side is equal to
$$f(w)=(f, \mathring K_w)_{\rm PW},$$ and since $f$ is arbitrary, we get
$$Lj^{-1}K_w=\mathring K_w.$$\end{proof}

In the formulas for the inverse spectral problem the last statement will only be used in the case $w=0$. We will
use the notations $k_t,\ \mathring k_t$ for the  reprokernels $K_0,\ \mathring K_0$ in $\BB_t$ and PW$_t$ respectively. Since
$$\mathring k_t(z)=\frac1\pi~\frac{\sin tz}z,$$
Theorem \ref{t2} implies
 \begin{equation}\qquad k_t=j_t\left[L^{-1}_{\mu,t} \mathring k_t\right].\label{eqSinc}\end{equation}

  \bs\subsection{Recovering the term $h_{11}(t)$} \label{sech11}
  In this section we obtain a formula which recovers the element $h_{11}$ of the Hamiltonian 
  $$\HH=\begin{pmatrix}h_{11}&h_{12}\\h_{12}&h_{22}\end{pmatrix}$$
   in CS \eqref{eq001} from
  the spectral measure.
 
  Note that for det-normalized CS with even spectral measures this amounts to a full solution for the inverse spectral problem as  $h_{12}\equiv 0$ by Theorem \ref{t3} and $h_{11}h_{22}=1$.
  As was mentioned before, systems with even measures represent Krein's string operators.

 As before, we assume that $\mu$ is a PW-sampling measure. Recall that by Theorem \ref{t1}, $h_{11}$ is determined by $\mu$ in a unique way (unlike $h_{22}$ or $h_{12}$).

  \begin{thm} Let $\mu\in\PW$ be the spectral measure of a system \eqref{eq001} with the Hamiltonian $\HH$. Then $t\mapsto k_t(0)$ is an absolutely continuous
  	function and 
  $$h_{11}(t)=h^\mu(t):=\pi\frac d{dt} k_t(0),$$
 where the kernel $k_t$ is given by \eqref{eqSinc}.
\end{thm}

  \begin{proof} From \eqref{eq007} we obtain the "Christoffel-Darboux" formula 
$$\int_0^t\langle\HH(s) X_z(s), X_w(s)\rangle~ds=\pi K_w(z,t).$$
It follows that
$$\langle\HH(t) X_z(t), X_w(t)\rangle=\pi \dot K_w(z,t).$$
If $z=w=0$, then $X_0(t)=(1,0)^{\scriptscriptstyle T}$, and
$$\langle\HH(t) X_0(t), X_0(t)\rangle=h_{11}(t).$$\end{proof}

\begin{remark} The nesting property $\BB_{t_1}\overset{\rm iso}\subset \BB_{t_2}$ for $t_t<t_2$, together with
$$k_t(0)=\sup\{f(0)|\ f\in\BB_t, ||f||_{\BB_t}=1\},$$ 
 implies that the function $t\mapsto k_t(0)$ is increasing, so the derivative in the definition of $h^\mu$ exists a.e.
 The same can be seen from our proof. The det-normalization implies $h^\mu>0$ a.e.

 $\triangle$
\end{remark}

\ss\begin{cor} If $\mu\in$PW is even, then $$\HH=\begin{pmatrix}h^\mu&0\\0&h_{\mu}^{-1}\end{pmatrix},$$
	where $h_\mu$ is from the last theorem,  is the unique det-normalized diagonal Hamiltonian such that $\mu$ is the spectral measure of \eqref{eq001}.\end{cor}

\bs\subsection{Conjugate kernels}\label{secCK} Let $\HH$ be a det-normalized PW Hamiltonian and let $\BB_t$ and $\tilde \BB_t$ be the corresponding chains of dB-spaces generated by HB functions $E_t=A_t-iC_t$ and $\ti E_t=B_t-i D_t$. We drop the subindex $t$ in our formulas when
$t$ is fixed.

As before, we denote by $K_t(z,w)$ and $\tilde K_t(z,w)$ the corresponding reprokernels, sometimes using abbreviated notations $K_w(z),\ \ti K_w(z)$ when $t$ is fixed. Due to a special role played by reprokernels at 0, we also use the notation
$$k_t(z):= K_t(z,0)=\frac1\pi\frac{C_t(z)}z.$$
We now introduce "conjugate" kernels $\ti L_w(z)\equiv \ti L_t(z,w)\in\ti \BB_t$ defined as
$$\tilde L_w(z)=\frac1\pi~\frac{[D(z)A(\bar w)-B(z) C(\bar w)]-1}{z-\bar w}.$$
In particular
$$\tilde l_t(z):= \tilde L_t(z,0)=\frac1\pi~\frac{D(z)-1}z.$$

As before, we denote by $X_z(t)$ and $Y_z(t)$ the solutions to CS \eqref{eq001} with Neumann and Dirichlet initial conditions correspondingly, 
$$X_z(t):=\begin{pmatrix}
	A_t(z) \\ C_t(z)
\end{pmatrix}, Y_z(t):=
\begin{pmatrix}
	B_t(z) \\ D_t(z)
\end{pmatrix}
.$$

 The following relation presents a special case of Lagrange
identities. 

\begin{lem}\label{t1500} $$\pi\tilde L_t(z,w)=(Y_z,X_w)_{L^2(\HH, [0,t])}$$\end{lem}

 To prove the lemma we first establish the following relation for
 the transfer matrix.
 
 \begin{lem}
 $$	\frac \partial{\partial t} [M_t^*(w)\Omega M_t(z)]=(z-\bar w)M^*_t(w)\Omega M_t(z)
 .$$ 
 \end{lem}

\begin{proof}
$$	\frac \partial{\partial t} [M_t^*(w)\Omega M_t(z)]=
 + 	M_t^*(w)\Omega\frac \partial{\partial t}  M_t(z)] +\frac \partial{\partial t} [M_t^*(w)\Omega ]M_t(z)=$$$$
 zM^*_t(w)\HH M_t(z) -\bar w M_t^*(w)\HH^* M_t(z)=(z-\bar w)M^*_t(w)\Omega M_t(z)
 .$$ 
\end{proof}

\begin{proof}[Proof of Lemma \ref{t1500}]
	$$(Y_z,X_w)_{L^2(\HH, [0,t])}=
	\int_0^t <\HH Y_z(s),X_w(s)>ds=$$$$=
	\int_0^t\left<\HH M_s(z)\begin{pmatrix} 0\\ 1 \end{pmatrix},M_s(w)\begin{pmatrix} 1\\ 0 \end{pmatrix}\right>ds=$$$$=
	\int_0^t\left<M^*_s(w)\HH M_s(z)\begin{pmatrix} 0\\ 1 \end{pmatrix},\begin{pmatrix} 1\\ 0 \end{pmatrix}\right>ds=
	$$$$
	\int_0^t \left<\frac{\frac \partial{\partial s} [M_s^*(w)\Omega M_s(z)]}{z-\bar w}\begin{pmatrix} 0\\ 1 \end{pmatrix},\begin{pmatrix} 1\\ 0 \end{pmatrix}\right>ds=$$$$=
	\frac{\frac \partial{\partial t} [M_t^*(w)\Omega M_t(z)]-1}{z-\bar w}
	=\frac{[D(z)A(\bar w)-B(z) C(\bar w)]-1}{z-\bar w}.$$
\end{proof}

\begin{cor} $$h_{12}(t)=\pi\frac d{dt}~\tilde l_t(0).$$\end{cor}
\begin{proof} If $z=w=0$, then
$$(Y_z,X_w)_{L^2(\HH, [0,t])}=\int_0^t\left\langle\HH\begin{pmatrix} 0\\1\end{pmatrix},\begin{pmatrix} 1\\0\end{pmatrix}\right\rangle=\int_0^t h_{12}.$$
\end{proof}

 Our next goal is to express the conjugate kernel $\tilde l_t$ in terms of the spectral measure which will allow us to recover the element $h_{12}$ of the 
 Hamiltonian via the last corollary, see Theorem \ref{t007}. 

\begin{thm}\label{t4} For each $t$ there exists a linear unitary operator
$$T_t:~\BB_t\to \tilde \BB_t$$
such that
$$\forall w\in \C,\quad K_w\mapsto \tilde L_w.$$
If $t_1<t_2$, then $T_{t_1}$ coincides with the restriction of $T_{t_2}$ to $\BB_{t_1}$.\end{thm}
\begin{proof} The  operator $T_t$ can be defined via the formula
$$\tilde W_t=T_t\circ \WW_t,$$
where $\WW_t: L^2(\HH, [0,t])\to \BB_t$ and $\tilde\WW_t: L^2(\HH, [0,t])\to \tilde\BB_t$ are the restrictions of the corresponding Weyl transforms.\end{proof}

\begin{remark} We emphasize the operators $T_t$ (and the spaces $\tilde\BB_t$) are not determined by the spectral measure; they also depend on the choice of $\HH$. If we replace $\HH$ by $\check\HH=\HH_k$ (see Remark \ref{remHk}), then we get
$$\check T_t=T_t+kI.$$

$\triangle$
\end{remark}

\bs\subsection{Hilbert transform}\label{secHT} For a $\Pi$-finite measure $\mu$ and $f\in L^2(|\mu|)$ we will use the notation
$$ K(f\mu)(z)=\frac1{\pi }\int\frac{f(s)~d\mu(s)}{s-z},$$
and
$$\KK\mu(z)=\frac1{\pi }\int\left[\frac1{s-z}-\frac s{1+s^2}\right]~d\mu(s),$$
where $z\in\C\setminus\R$. If $f\in L^2(\mu)$ is an entire function, then we define
$$H^\mu f=K(f\mu)-f\KK\mu.$$
It is clear that $H^\mu f$ extends to an entire function:
$$
(H^\mu f)(z)=\frac1\pi\int\left[\frac{f(s)-f(z)}{s-z}+\frac{s f(z)}{1+s^2}\right]~d\mu(s).$$

\begin{example} If $\mu=m$, then  $H^\mu f$ is the analytic continuation of the Hilbert transform
$$Hf(z)=\frac1\pi {\rm v.p.}\int_\R\frac{f(s)ds}{s-z}$$
and
$$H^\mu:~{\rm PW}_t\to  {\rm PW}_t$$
is a unitary operator.
\end{example}

\begin{remark}\label{rem007}
	If $\mu$ is a periodic locally finite measure then
	$$\int\frac{s}{1+s^2}~d\mu(s)=c$$
	exists and
	$H^\mu$ takes on the form
	$$H^\mu f=Kf\mu-fK\mu +cf=\frac1\pi\int\frac{f(s)-f(z)}{s-z}d\mu(s)+cf.$$

$\triangle$
\end{remark}


\begin{thm}\label{t5}  Let $\HH$ be a det-normalized PW Hamiltonian, and let $\mu$ be the spectral measure. Let $\BB_t,\ \ti\BB_t$ be the corresponding
	dB-spaces and let $T_t$ be the unitary operators from Theorem \ref{t4}. Then there exists a real constant $c=c(\HH)$ such that
$$\forall t,\qquad T_t=(H^\mu-cI)\Big|_{\BB_t}.$$\end{thm}

\begin{remark}Note that $c(\HH_k)=c(\HH)-k$, where $\HH_k$ is like in Remark \ref{remHk}.

$\triangle$
\end{remark}

\begin{proof} $ $
	
	1) Let us fix $t>0$.
Let $\nu=\mu_{-1}$ be the representing measure for the space $\BB_{t_0}$
for some $t_0>t$ and $\a=-1$, as defined in Section \ref{secR}.
Let us first show that 
$$T_t=(H^\nu-cI)\Big|_{\BB_t}.$$
Indeed, for $F\in \BB_t$,
$$T_tF(z)=<T_t F,\ti K_z>_{\ti \BB_t}=<T_t F, T_t L_z>_{\ti \BB_t}=
=<F, L_z>_{ \BB_t}=$$$$
\int F(s)\frac{[A(s)D(z)-C(s)B(z)]-1}{s-z}d\nu(s)=$$$$
=-K(F\nu)(z)-B(z)\int \frac{F(s)C(s)}{s-z}d\nu(s),$$
because by construction $\nu=\mu_{-1}$ is supported at the zeros of $A$.
On the other hand,
$$F(z)=<F,K_z>_{\BB_t}=-\int F(s)\frac{A(s)C(z)-C(s)A(z)}{s-z}
d\nu(s)=$$$$A(z)\int \frac{F(s)C(s)}{s-z}
d\nu$$
 Therefore 
$$T_t F= -K(F\nu) -\frac BA F.$$
It is left to recall that by Lemma \ref{lemBA}
$$\KK\nu=\KK\mu_{-1}=-\frac BA+\const.$$

2) We will now show that $H^\nu=H^\mu$ on $\BB_t$.

Fix any function $G\in \BB_t$ such that $G(i),G(-i)\neq 0$. Consider the expression
\begin{equation}
	\int\left[\frac{F(s)-F(z)}{s-z}-\frac 12\left(\frac{F(z)}{G(i)}\frac{G(s)-G(i)}{s-i}+\frac{F(z)}{G(-i)}\frac{G(s)-G(-i)}{s+i}\right)\right]d\nu(s).
	\label{e3}
\end{equation}
Let us denote the function under the integral by 
$S(z,s)$. Note that since $\BB_t$ is regular, for each fixed $z$ the functions 
$S(z,w)$ and $wS(z,w)$ belong to $\BB_t$ as functions of $w$.
Hence, the integral must remain the same if $\nu$ is replaced
with any measure $\eta$ such that $\BB_t\overset{\rm iso}\subset L^2(\eta)$.
Indeed, let $Q\in \BB_t$ be any  function such that $Q(0)=1$.
Once again, since $\BB_t$ is regular, $R(z)=(Q(z)-1)/z\in \BB_t$.
Therefore
$$ \int S(z,s)d\eta(s) =\int S(z,s)(\bar Q(s)-s\bar R(s))d\eta(s)=$$$$
<S(z,s),Q(s)>_{\eta}-<sS(z,s),R(s)>_{\eta}=$$$$
<S(z,s),Q(s)>_{\nu}-<sS(z,s),R(s)>_{\nu}.$$
Since $\mu$ is the spectral measure, it can replace $\nu$
in \eqref{e3}. It is left to notice that
 when $\nu=\mu$ the expression in \eqref{e3} is equal to
$$KF\mu(z)-F(z)\KK\mu(z)+F(z)\int\left[\frac{1}{G(i)}\frac{G(s)}{s-i}+\frac{1}{G(-i)}\frac{G(s)}{s+i}\right]d\mu.$$

Recall that $G$ was a fixed function, independent of $F$. It remains to put
$$c_\mu=d_\mu-d_{\nu},$$
where $d_\mu$ is
equal to the last integral and $d_{\nu}$ is
equal to the last integral with $\mu$ replaced by $\nu$.
\end{proof}

\bs\subsection{Recovering the term $h_{12}(t)$}\label{secGmu}

\begin{thm}\label{t007} Let $\mu\in\PW$, and let the functions $k_t\in\PW_t(\mu)$ be obtained from $\mu$ by via the formula
	\eqref{eqSinc}. Define $\tilde l_t=H^\mu k_t$. Then $\mu$ is the spectral measure of the Hamiltonian
$$\HH=\begin{pmatrix}h^\mu&g^\mu\\g^\mu&\frac{1+g^2_{\mu}}{h^\mu}\end{pmatrix},$$
where
$$g^\mu(t):=\pi\frac d{dt} \tilde l_t(0).$$
\end{thm}

\begin{proof} Let $\check \HH$ be the Hamiltonian provided by Theorem \ref{t1}.
	Uniqueness of the regular chain of dB-spaces in $L^2(\mu)$ implies
	that $\check \BB_t=\PW_t(\mu)$.
	 Let $\check T_t$ be the operators from Theorem \ref{t4}. Denote
$$\check l_t=\check T_tk_t.$$
By Theorem \ref{t5} there is a real number $c$ such that $\check T_t=H^\mu-cI$ on $\BB_t$, so
$$\check l_t=\tilde l_t-c k_t.$$
Differentiating, we get
$$\check h_{12}=g^\mu-ch^\mu.$$
It follows that
$$\HH=T_c^*\check \HH T_c.$$ \end{proof}

\bs\subsection{Equations for the Fourier transform of $k_t$}\label{secFt} How to compute the functions $h^\mu$ and $g^\mu$
from Theorem \ref{t5}? Sometimes  it is helpful to work with the functions
 $$\psi_t:=\hat k_t,$$
so that
$$\frac1{\sqrt{2\pi}}\int_\R\psi_t=k_t(0).$$

  
   If $f\in {\rm PW}_t$, then
  $$f(0)=\frac1{\sqrt{2\pi}}\int_{-t}^t\hat f(\xi)~d\xi$$
 while  also
  $$ f(0)=(f, \mathring k_t)_{PW_t}=\left(\hat f,\FF(\mathring k_t)\right)_{L^2(-t,t)}.$$
This leads us to the well known formula for the Fourier transform of the sinc function:
   $$\FF \mathring k_t =\FF\left(\frac{\sin tz}{\pi z}\right)
=  \frac1{\sqrt{2\pi}}~1_{(-t,t)}.$$

The Fourier transform $\psi_t$ of the reprokernel $k_t$ can be found using the following statement.

\begin{thm}\label{tTT} $\psi=\psi_t$ satisfies
$$\psi\ast \hat\mu=1 \quad  {\rm on}\quad (-t,t)$$
and
$$\psi=0 \quad  {\rm on}\quad \R\setminus [-t,t].$$\end{thm}

\begin{proof} We denote by $P_t$ the orthogonal projection $P_t:L^2(\R)\to \PW_t$.  Somewhat formally we interpret the equation $L_{\mu} k_t=\mathring k_t$ as
$$P_t(k_t\mu )~=~\mathring k_t,$$
and taking Fourier transforms we obtain
$$1_{(-t,t)}~\FF(k_t\mu)~=~\frac1{\sqrt{2\pi}} 1_{(-t,t)}.$$
Note that
$$\FF(k_t\mu)~=~\frac1{\sqrt{2\pi}} \hat\mu\ast \psi_t.$$
\end{proof}

\begin{remark}
	For $\mu\in\PW$ and $f\in\PW_a$ the map
	$$g\mapsto\int g\bar f d\mu$$
	defines a bounded linear functional on any $\PW_t,\ t>0$.
	It follows that the restriction of $\hat f\ast \hat\mu$ 
	to any interval $(-t,t)$ belongs to $L^2([-t,t])$.
	In particular, so does  $\psi_t\ast \hat\mu$ because $k_t\in \PW_t(\mu)\doteq \PW_t$.

$\triangle$
\end{remark}

\bs\section{Different initial conditions}\label{sech22}

\bs\subsection{The term $h_{22}(t)$} Let $\HH$ be a Hamiltonian such that $\det(\HH)=1$ a.e. and let $\tilde \mu$ be the spectral measure of $\HH$ for the initial condition $(0,1)^{\scriptscriptstyle T}$. Assume that $\tilde \mu\in$(PW). Then we can recover the term $h_{22}$ from $\tilde \mu$ exactly as we did it for $h_{11}$ and the spectral measure $\mu$:
$$h_{22}(t)=\pi\frac d{dt} \tilde k_t(0), \qquad \tilde k_t:=j_t\left[L^{-1}_{\tilde \mu,t}\mathring k_t\right]$$
or simply,
$$h_{22}=h^{\tilde \mu}.$$

\begin{thm}\label{t8} Let $\mu$ and $\tilde\mu$ be the spectral measures of the det-normalized Hamiltonian $H$ corresponding to the Neumann and Dirichlet 
	boundary conditions at 0 correspondingly. Assume that both $\mu$ and $\ti\mu$ are PW-sampling measures. Then there exists a real number $c$ such that
$$\HH=\begin{pmatrix}h^\mu&g^\mu-ch^\mu\\g^\mu-ch^\mu&h^{\tilde \mu}\end{pmatrix},$$
\end{thm}

 \begin{cor} If $\HH$ is diagonal and both $\mu$ and $\tilde\mu$ are PW-sampling measures, then
$$h^\mu h^{\tilde\mu}=1 \qquad {\rm a.e. ~ on}\quad\mathbb R_+.$$
\end{cor}

\bs\subsection{AC duality}\label{secCD} In view of the last theorem,  we wold like to describe all possible measures $\tilde \mu$ for a given $\mu$.

As was defined in Section \ref{secClark},  $\tilde \mu $ is AC-dual to $\mu$ if there exists a function $\phi\in H^\infty(\C_+)$, $\|\phi\|_\infty\le1,$ such that
$$\mu=\sigma_{-1}^\phi,\qquad \tilde\mu=\sigma_{1}^\phi,$$
 i.e.,
\begin{equation}\PP\mu=\Re\frac{1-\phi}{1+\phi},\qquad \PP\tilde\mu=\Re\frac{1+\phi}{1-\phi}.\label{eq010}\end{equation}
For simplicity we'll assume  $\mu(\infty)=\tilde\mu(\infty)=0$; as we mentioned earlier   our spectral measures will satisfy this condition. Given $\mu$ there is a 1-parameter family of dual measures $\tilde\mu=\sigma_{1}^{\phi_b}$ where
\begin{equation}\phi_b=\frac{i-\KK\mu-b}{i+\KK\mu+b}.\qquad (b\in\R).\label{eqPhib}\end{equation}

We will sometimes use the notation $\ti\mu_b$ if $\mu=\s^{\phi_b}_{-1}$,
$\ti\mu_b=\s^{\phi_b}_1$ for $\phi_b$
as in \eqref{eqPhib}. In particular, 
\begin{equation} \PP\ti\mu_b=\Re\frac i{\KK\mu+b}.\label{eqDM}\end{equation}

\begin{thm} Let $\mu\in\PW$ be given. Then $\tilde\mu$ is an AC-dual measure of $\mu$ iff there exists a Hamiltonian $\HH$ such that $\mu$ and $\tilde\mu$ are the spectral measures of $\HH$ with the Neumann and Dirichlet initial conditions respectively.\end{thm}

\begin{proof} (i) Let $\mu$ and $\tilde \mu$ be the spectral measures of some $\HH$, and let $\begin{pmatrix}A&B\\C&D\end{pmatrix}$  be the transfer matrix of $\HH$. For each $t>0$ we define the measure
	$\mu_t\equiv \mu^{E_t}$ as the  representing measure
	for the space $\BB(A_t-iC_t)$ supported on $\{A_t=0\}$ and
	$\tilde \mu_t\equiv \mu^{\tilde E_t}$ as the  representing measure for the space $\BB(B_t-iD_t)$ supported on $\{B_t=0\}$. If we denote
	$$\phi_t=\frac{A_t-iB_t}{A_t+iB_t},$$
	then
	$$\mu_t=\sigma_{-1}^{\phi_t},\qquad \tilde\mu_t=\sigma_{1}^{\phi_t}$$
	(note that  $\KK\mu_t=-B/A$ like in Lemma \ref{lemBA}).
	Since $\mu_t$ is a representing measure for $\BB_t=\BB(E_t)$, we
	have 
	$$\BB_s\overset{\rm iso}\subset L^2(\mu_t)\ {\rm for\ all}\ s\leq t.$$
	Since $\mu$ is the unique measure for which the above 
	inclusion holds for all $s>0$ and since
	$\cup_t\BB_t$ is dense in $L^2(\mu)$, we obtain that $\mu_t\to\mu$, $*$-weakly as $t\to\infty$. Similarly,
	$\ti\mu_t\to \ti\mu$. Since
	$$\phi_t=\frac{1+i\KK\mu_t}{1-i\KK\mu_t},$$
	the functions $\phi_t$ tend to $\phi$,
	$$\phi=\frac{1+i\KK\mu}{1-i\KK\mu}$$
	pointwise in $\C_+$. Then $\mu$, $\ti\mu$ and $\phi$
	satisfy \eqref{eq010} and we see that $\mu$ and $\tilde\mu$ are dual.
	
	(ii) Consider $\check\HH=\HH_k$ defined as in Remark \ref{remHk}. Then
	$\check A=A$ and $\check B=B+kA$, so
	$$\check\phi_t=\frac{A_t-ikA_t-iB_t}{A_t+ikA_t+iB_t}.
	$$
	Recalling that, according to Lemma \ref{lemBA}, $\KK\mu=-B/A+\const$ and using \eqref{eqPhib}, we obtain all AC-dual measures as $k$ ranges over $\R$.
\end{proof}

\bs\subsection{PW-sampling in AC-families
}\label{secACDPW}
The statement of the Theorem \ref{t8} includes an assumption
that both measures $\mu$ and $\ti\mu$ are PW-sampling.
This brings up a natural question of whether this assumption
is redundant, i.e, if $\mu\in\PW$ implies $\ti\mu\in\PW$.

Our next goal is to study this question in more detail. First we give an example showing that the answer, in general, is negative. 

We denote by ${\rm sign }\ x$  the sign-function defined as $\pm 1$
on $\R_{\pm}$ and as 0 at 0.

\begin{example}\label{exnPW}
	Consider $\mu(x)=2+{\rm sign }\ x$.
 Clearly, $\mu\in\PW$.
	Let $\phi\in H^\infty,\ ||\phi||_\infty\leq 1$ be the
	function such that $\phi(i)\in\R$ and \eqref{eqClark} is satisfied for
	$\mu$ and $\ti\mu$ with $\a=\mp 1$ respectively.
	Then
	$$\KK\mu=i\frac {1-\phi}{1+\phi}=-\frac 1{\KK\ti\mu}.$$
	Since $\KK\mu(z)\to\infty$ as $z\to \infty$ the boundary
	values of 	$\KK\ti\mu(z)$ on $\R$ tend to 0 at $\pm\infty$. It follows that 
	$$\ti\mu((x,x+1))\leq \int_x^{x+1} |\KK\ti\mu(t)|dt\to 0\ {\rm as}\ x\to\pm\infty.$$
	As was discussed in Section \ref{secPWs}, this implies $\ti\mu\not\in\PW$.
\end{example}

\begin{remark}
	This example has the following interpretation: even though $\ti \mu$ is not a PW-measure,
	the corresponding inverse spectral problem can be solved by switching to its dual $\mu\in\PW$.
	
	If, on the other hand, all (or two) of the measures in an AC-family are PW-sampling then a combination of 
	Theorems \ref{t007} and \ref{t8}  provides a continuum of often rather non-trivial identities between solutions
	to the inverse spectral problems corresponding to different measures from the family.
	
	$\triangle$
\end{remark}

We now formulate a necessary and sufficient condition
for all measures dual to a given $\mu\in\PW$ to be PW-sampling.

\begin{thm}\label{t9}
	Let $\mu\in \PW$. All measures dual to $\mu$ are PW-sampling if and only if there exists $d\neq 0$ such that
	$\KK\mu(z)$ is bounded on the line $\{\Re z=d\}$.
\end{thm}

Note that $\KK\mu$ is bounded on some line $\{\Re z=d\},\ d\neq 0$, iff $\KK\mu$ is bounded on every such line.

	 First we establish the following
	 
	\begin{lem}
		Let $\nu$ be a positive measure on $\R$,
		\begin{equation}\frac 1\pi\int\frac{d\nu(x)}{1+x^2}\leq C,\label{e1}\end{equation}	
		and let $I_1,I_2,I_3$ be three  adjacent  intervals,   enumerated
		left to right, $|I_j|=L>0$. Let $I=I_1\cup I_2\cup I_3$ be an open interval. Suppose that 
		
		1) $|\KK\nu|<C$ on $\R+i$
		
		2) $\nu(I_1), \nu(I_3)>\e$.
		
		Let $\ti\nu=\ti\nu_b,\ b\in\R$ be a dual measure.
		Then $\ti\nu(I)>d$, where $d>0$ depends only on $C, L,\e$ and $b$, but not on the position of the intervals on $\R$.
		
	\end{lem}
	
	\begin{proof}
		
		First, let  $I$ be an interval centered at zero (i.e., we do not aim to prove that $d$ does not depend on the position of $I$ at this stage). Denote by $d_0$ the infimum of $\ti\nu(I)$ taken
		over all $\nu$ satisfying the conditions. Suppose that $d$ is zero and that $\nu^n$ is a sequence
		of measures such that $\ti\nu^n_b(I)\to 0$. Choosing a subsequence
		of $\nu^n$ if necessary, we assume that $\nu^n$ converges
		$*$-weakly in the space of $\Pi$-finite measures on $\hat\R$ to a $\Pi$-finite measure $\eta$.
		Note that $\eta$ must still satisfy 2), and therefore is non-trivial.
		The relation between Cauchy integrals for  dual measures implies $\ti\nu^n_b\to\ti\eta_b$.
		Since $I$ is open,
		$\ti\eta(I)=\lim \ti\nu^n(I)=0$.
		Since $\ti\eta$ is a positive measure vanishing on $I$,
		$\KK\ti\eta$ increases monotonically on $I$. It follows that
		$\eta$ may have only one point-mass on $I$, which contradicts
		2). Therefore $d>0$. 
		

		Now let $I$ be an arbitrary interval and let us show that because of 1), $d$ does not depend on its position on $\R$. 
		
		Suppose
		that there exists a sequence 
		$I^n$ of open intervals and a sequence $\mu^n$ of measures such that each of them 
		satisfies the conditions of the lemma in place
		of $\mu$ and $I$, but $\ti\mu^n_b(I^n)\to 0$.  Suppose that 
		$I^n$ is centered at $x^n$. Consider the
		sequence of measures $\nu^n=\mu^n(x-x^n)$.
		Since
		$$\KK\nu^n(z)=\KK\mu^n(z-x^n)+b^n,$$
		where
		$$b^n=\int \frac t{1+t^2}d\nu(t)-
			\int \frac {t-x^n}{1+(t-x^n)^2}d\nu(t)=
				\Im \KK\nu(i)-\Im \KK\nu(i+x^n),$$
				we obtain
				$\ti\nu^n_{b-b^n}=\ti\nu_{b}(z-x^n)$.
				Note that because of 1), $|b^n|<2C$.
				
				Choosing a subsequence if necessary, we can assume that $b^n\to b_0$.
				It follows from \eqref{eqPhib} and \eqref{eq010} that 
				$\PP\nu^n_{b-b_0}-\PP\nu^n_{b-b^n}\to 0$ pointwise in $\C_+$. Hence $\nu^n_{b-b_0}-\nu^n_{b- b_n}\to 0$, $*$-weakly
				and $\nu^n_{b-b_0}(I)-\nu^n_{b- b_n}(I)\to 0$.
				By our choice of the measures $\nu^n=\mu^n(x-x^n)$,
					$\ti\nu^n_{b-b^n}(I)\to 0,$
					which implies  $\nu^n_{b-b_0}(I)\to 0$. Since $\nu^n$ satisfy the conditions of the lemma on the interval $I$ centered at 0, this contradicts
				the first part of the proof.
		\end{proof}

\begin{proof}[Proof of Theorem \ref{t9}]
	
	Let $\mu\in \PW$ satisfy the conditions of the theorem. 
	Then  $\ti\mu((x,x+1))$ is uniformly bounded. Indeed, if it is not, then $P\ti\mu$
	is unbounded on $\R+id$. Let $x_n\in \R$ be a sequence of
	points such that $P\ti\mu(x_n+id)\to\infty$. It follows that
	$\phi(x_n+id)\to 1$ and 
	 $P\mu(x_n+id)\to 0$. Then
	one can choose intervals $J_n$ centered at $x_n$ such that $|J_n|\to \infty$ and $\mu(J_n)\to 0$.
	Existence of such intervals contradicts $\mu\in \PW$ because it implies
	$$||f(x-x_n)||_{L^2(\mu)}\to 0$$
	for any $f\in\PW_a$.
	
	Let $I_1, I_2,I_3$ be consequent $\delta$-massive intervals for $\mu$, from the statement of Theorem \ref{PWmu}. Then by the last lemma $I=I_1\cup I_2\cup I_3$ is a $\delta'$-massive interval for $\ti\mu$
	for some $\delta'>0$.
	Since  $\mu$ is sampling for every $\PW_{3a}$, $\ti\mu$ is sampling for every $\PW_a$. 
	
	In the opposite direction, suppose that $K\mu$ is unbounded on some $\R+i\e$. Similar to above, $P\ti\mu$ is
	then not bounded from below by a positive constant, which implies the existence of intervals $J_n$ such that $|J_n|\to \infty$ but $\ti\mu(J_n)\to 0$.
	Hence, $\ti\mu\not\in \PW$. 
\end{proof}

  \bs\section{Periodic spectral measures}

  \bs\subsection{Measures from $M^+(\T)$}\label{secPer} Let $M^+(\T)$ be the set of all finite positive measures on the unit circle $\T\sim [0,2\pi]$ with {\it  infinite support}. We  identify measures on the circle with $2\pi$-periodic measures on $\R$,
  $M^+(\T)\subset M^+(\R)$, and define $M^+_{\rm even}(\T)\subset M^+(\R)$ as the subset of even periodic measures.

Notice that all measures in $M^+(\T)$ are $\Pi$-finite. Furthermore, as was mentioned before, it follows from Theorem \ref{PWmu} that 
$M^+(\T)\subset\PW$. Let us list two more important properties of $M^+(\T)$
pertaining to our problems.

  \begin{prop} If $\mu\in M^+(\T)$, then all dual measures are in $M^+(\T)$\end{prop}
  
  \begin{proof}
  	Periodicity of $\ti\mu_b$ follows from \eqref{eqPhib} and \eqref{eq010}.
  	A well-known property of dual measures is that supports of $\mu$ and $\ti\mu$ alternate, i.e., between any two points of one support there is a point of another. It follows that, since the support $\mu$ is infinite, so is the support of $\ti\mu$.
  \end{proof}

  \begin{prop} If $\mu\in M^+_{\rm even}(\T)$, then there is a unique  $\tilde \mu\in M^+_{\rm even}(\T)$ such that $\mu$ and $\tilde \mu$ are dual: $\ti\mu=\ti\mu_b,\ b=-\Re\KK\mu(0)$.\end{prop}
  
  \begin{proof}
  	From \eqref{eq010} we see that in order for $\ti\mu_b$ to be even, $\Re \frac{1+\phi_b}{1-\phi_b}$ must be even. From \eqref{eqDM} we obtain
  that 
  $$\PP\ti\mu_b=\Re \frac i{\KK\mu+b}=\frac {\PP\mu}{|\KK\mu+b|^2}.$$
  In the last expression, $\PP\mu$ is even because $\mu$ is even. In order
  for the whole fraction to be even, $|\KK\mu+b|$ must be even. 
   If $b=-\Im \KK\mu(0)$, then 
  $$\Im(\KK\mu +b)=\Im \KK\mu=\PP\mu$$ 
  is even, $\Re(\KK\mu+b)$ is odd and  $|\KK\mu+b|$ is even.
  For any other value of $b$, $\Re(\KK\mu(0)+b)\neq 0$,
   $\Re(\KK\mu+b)$ is not odd and  $[\Re(\KK\mu+b)]^2$ is not even. Hence,
  $|\KK\mu+b|$ is not even.
  \end{proof}

Note that for the unique even dual measure $\ti\mu$ we have
\begin{equation}\KK\ti\mu =-\frac 1{\KK\mu},\ \PP\ti\mu=-\Re \frac 1{\KK\mu}.
	\label{eqEven}\end{equation}

  \bs\subsection{Moments}\label{secMom} If $\mu\in M^+(\T)$, we can consider the Fourier series
  $$\mu\sim\sum_{-\infty}^\infty \gamma_k e^{ikx},$$
  where $$\gamma_k=\frac1{2\pi}\int_0^{2\pi}e^{-ikx}~d\mu(x)=\a_k+i\b_k$$
  are the "trigonometric moments" of $\mu$. Clearly,
  $$\gamma_{-k}=\overline{\gamma_k},$$
  so
  $$\mu\sim \gamma_0+2\sum_1^{\infty}(\alpha_k\cos kx-\beta_k\sin kx),$$
  and $\mu$ is even iff all moments are real, i.e., all $\beta_k=0$.

   The moments can be arranged into the Toeplitz matrix
  $$\Gamma(\mu)=\begin{pmatrix}\gamma_0&\gamma_1&\gamma_2&\dots\\
  \gamma_{-1}&\gamma_0&\gamma_1&\dots\\
  \gamma_{-2} &\gamma_{-1}&\gamma_0&\dots\\
  \dots&\dots&\dots&\dots
  \end{pmatrix}$$
  We will denote by $\Gamma_n(\mu)$ the  $n\times n$  matrix
  in the upper left corner of $\Gamma$. It is well known that a sequence $\{\gamma_k\}$ is the sequence of moments of some $\mu\in M^+(\T)$ iff
  $$\forall n, \qquad \det \Gamma_n>0.$$

   Finally let us recall the relation between the moments of a periodic measure and its  Fourier transform: if $\mu\sim\sum \gamma_k e^{ikx}$, then
  $$\hat\mu=\sqrt{2\pi}\sum_{-\infty}^\infty \gamma_k\delta_k.$$

  \bs\subsection{Computation of $h^\mu(t)$}\label{secHmu} Recall that  if $\mu\in\PW$, then
  $$h_\mu(t):=\pi\frac d{dt}\left[L_{\mu,t}^{-1} \stackrel \circ{k_t}(0)\right]$$
  
  If $A=(a_{jk})$ is a matrix we use the notation $\SS[A]$ for the sum
  of the elements of $A$:
  $$\SS[A]=\sum_{ j,k} a_{jk}.$$

  \begin{thm}\label{tPer} If $\mu\in M^+(\T)$, then the function $h^\mu(t)$ is locally constant on $\R_+\setminus \frac12\mathbb N$, i.e., $$h_\mu(t)=h_0,\; h_1, \; h_2,\;\dots\quad {\rm on}\quad \left(0,\frac12\right),\; \left(\frac12,1\right),\; \left(1,\frac32\right),\;\dots,$$ and
  $$h_n=\SS\left[\Gamma_{n+1}(\mu)^{-1}\right]-\SS\left[\Gamma_{n}(\mu)^{-1}\right],\quad h_0=\SS[\Gamma_1^{-1}]=\gamma_0^{-1}.$$
  \end{thm}

 \begin{proof} Since
$$\hat
\mu=\sqrt{2\pi}\sum \gamma_k\delta_k,$$
 the Fourier transform $\phi_t$ of $k_t$ satisfies the relations 
in Theorem \ref{tTT}.
Let
$$f=f_t~=~\sqrt{2\pi}~\psi_t.$$
Then $f$ satisfies
\begin{equation}\sum \gamma_k f(\xi-k)=1 \quad  {\rm on}\quad (-t,t)
\quad {\rm and}\quad f=0 \quad  {\rm on}\quad \R\setminus [-t,t].\label{eqTT1}\end{equation}
Recall
$$h^\mu(t)=\pi\frac d{dt}\frac1{\sqrt{2\pi}}\int_{\R}\psi_t= \frac 12 \frac d{dt}\int_{\R}f_t.
$$

Let us now find $f_t$ from \eqref{eqTT1}.

 Case $0<t<\frac12$. 
If $\xi\not\in (-t,t)$, then $f(\xi)=0$. If $\xi\in(-t,t)$, then $\xi\pm1, \xi\pm 2, \dots$ are not in $(-t,t)$, so
$$\sum \gamma_k f(\xi-k)=\gamma_0f(\xi)\quad {\rm must~be}\; =1.$$
Thus
$$f_t=\gamma_0^{-1}1_{(-t,t)},\qquad \int f_t=2t\gamma_0^{-1},\qquad h^\mu(t)=\gamma_0^{-1}.$$

\ms\no
Case $\frac12<t<1$. Consider the partition of the interval $(-t,t)$ by the points from $\pm t+\Z$:
$$-t<t-1<-t+1<t.$$ We claim that $f=f_t$ is constant on each interval of the partition. For instance, let $\xi\in(-t,t-1)$. Then $f(\xi+k)=0$ unless $k=0$ or $k=1$, so equating the value of $f\ast\hat\mu$ to 1 on the first and third intervals
of the partition we get
$$\gamma_0f(\xi)+\gamma_{-1} f(\xi+1)=1$$
and 
$$\gamma_0f(\xi+1)+\gamma_{1}f(\xi)=1.$$
The linear  system for $f(\xi)$ and $f(\xi+1)$ have the matrix $\Gamma_2$, which
implies 
$$\begin{pmatrix}
	f(\xi)\\f(\xi+1)
\end{pmatrix}=\Gamma_2^{-1}
\begin{pmatrix}
	1\\1
\end{pmatrix}.$$
In particular
 it follows that the values $f(\xi)$ and $f(\xi+1$ are constant and 
$$f(\xi)+f(\xi+1)=\SS[\Gamma_2^{-1}].$$
(Here we use the formula $\SS[M]=M\vec{1}\cdot \vec1$, where
$\vec{1}=(1,1,...,1)^{\scriptscriptstyle T}$.) 

On the other hand, if $\xi$ is in the middle interval $(t-1,-t+1)$, then $$f(\xi)=\gamma_0^{-1}=\SS[\Gamma_1^{-1}]$$ by the same argument as in the previous case. Note that the middle interval is shrinking with velocity 2, while the two other intervals are expanding with velocity 2. Integrating $f$ from 
$-t$ to $t$ and then  differentiating the result with respect to $t$ we get
$$h^\mu(t)=
\SS[\Gamma_2^{-1}]-\SS[\Gamma_1^{-1}].$$

Let us now consider the general case 
$\frac n2<t<\frac {n+1}2$. 

There is a simple way to visualize  the partition
of the interval $(-t,t)$, see the picture below. In the $(\xi, t)$-plane 
consider  the interval (blue) on the  line
$t=\const$ lying in the sector $t>|\xi|$. The points of the partition
are given by intersections of the interval with the lines $t=\pm \xi +n,\ n\in\Z$, which form a square lattice in the sector.

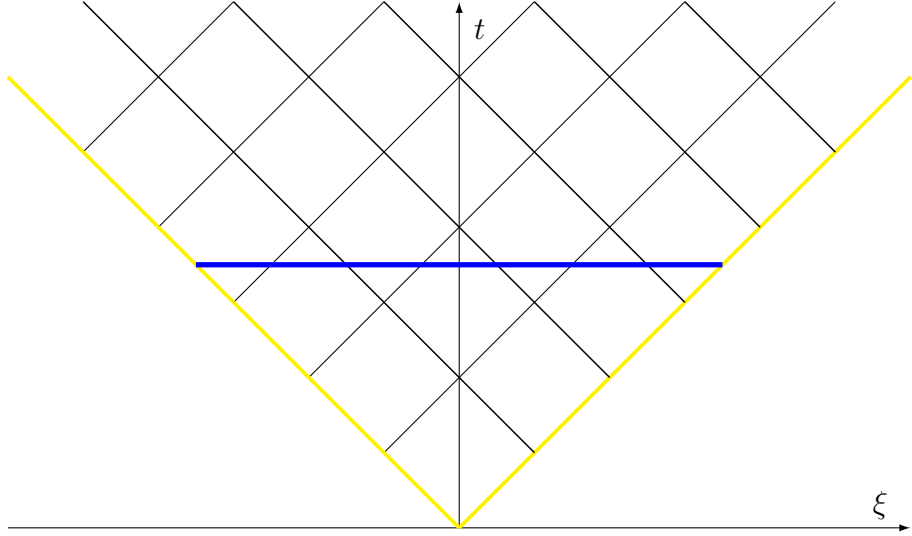
\setlength{\unitlength}{1cm}
\begin{picture}(6,7)(-2,0) 
	
	\put(-2,0){\vector(1,0){12}}
	\put(9.5,0.2){\hbox{\texttt{$\xi$}}}
	\put(4,0){\vector(0,1){7}}
	\put(4.2,6.5){\hbox{\texttt{$t$}}}
	
	\put(-1,5){\line(1,1){2}}
	\put(0,4){\line(1,1){3}}
	\put(1,3){\line(1,1){4}}   
	\put(2,2){\line(1,1){5}}
	\put(3,1){\line(1,1){6}}
	\linethickness{0.5mm}
	\put(4,0){\color{yellow}\line(1,1){6}}
	\linethickness{0.2mm}
	
	\put(9,5){\line(-1,1){2}}
	\put(8,4){\line(-1,1){3}}
	\put(7,3){\line(-1,1){4}}
	\put(6,2){\line(-1,1){5}}
	\put(5,1){\line(-1,1){6}}
	\linethickness{0.5mm}
	\put(4,0){\color{yellow}\line(-1,1){6}}
	\linethickness{0.2mm}
	
	\linethickness{0.7mm}
	\put(0.5,3.5){\color{blue}\line(1,0){7}}
	

	
\end{picture}

If we enumerate the $2n+1$ intervals of the partition from left to right, then odd numbered intervals
$$I_1=(-t,t-n),\ I_3=(-t+1,t-(n-1)),...,\ I_{2n+1}=(-t+n,t)$$
are expanding while the even numbered intervals
$$I_2=(t-n,-t+1),\ I_4=(t-(n-1), -t+2),...,\ I_{2n}=(t-1,-t+n)$$
are shrinking. If $\xi\in I_1$ then $\xi+k\in I_{2k+1},\ k=1,2,..,n$. Equating 
the values of $f\ast\hat\mu$ on each odd interval to 1 we obtain the system of $n+1$ linear equations
\begin{equation}\begin{pmatrix}
	f(\xi)\\f(\xi+1)\\ \vdots \\ f(\xi+n)
\end{pmatrix}=\Gamma_{n+1}^{-1}
\begin{pmatrix}
	1\\1\\\vdots \\1
\end{pmatrix},$$
from which we deduce
$$f(\xi)+...+f(\xi+n)=\SS [\Gamma_{n+1}^{-1}]\label{eq1500}\end{equation}
for the values on the expanding intervals.
Similarly for $\xi\in I_2$,
\begin{equation}\begin{pmatrix}
	f(\xi)\\f(\xi+1)\\ \vdots \\ f(\xi+(n-1))
\end{pmatrix}=\Gamma_{n}^{-1}
\begin{pmatrix}
	1\\1\\\vdots \\1
\end{pmatrix},$$
which yields 
$$f(\xi)+...+f(\xi+(n-1))=\SS [\Gamma_n^{-1}]\label{eq1501}\end{equation}
for the values on the shrinking intervals. Integrating $f$ from $-t$ to $t$ and differentiating the result
with respect to $t$ we obtain
$$h^\mu(t)=
\SS[\Gamma_{n+1}^{-1}]-\SS[\Gamma_n^{-1}]$$
 as claimed.
\end{proof}

  \bs\subsection{Computation of $g^\mu(t)$}\label{secGmu1} Let $\mu\in M^+(\T)$ and let
  $$\Delta(\mu)=\begin{pmatrix}0&\gamma_1&\gamma_2&\dots\\
  -\gamma_{-1}&0&\gamma_1&\dots\\
  -\gamma_{-2} &-\gamma_{-1}&0&\dots\\
  \dots&\dots&\dots&\dots
  \end{pmatrix}$$
where, as before, $\gamma_k$ are trigonometric moments of $\mu$. We denote
by $\Delta_n$ the $n\times n$ matrix in the upper left corner of $\Delta$.

  \begin{thm} If $\mu\in M^+(\T)$, then the function $g^\mu(t)$ is locally constant on $\R_+\setminus \frac12\mathbb N$, i.e. $$g^\mu(t)=g_0,\; g_1, \; g_2,\;\dots\quad {\rm on}\quad \left(0,\frac12\right),\; \left(\frac12,1\right),\; \left(1,\frac32\right),\;\dots$$ and
  $$g_n=\SS\left[\Delta_{n+1}\Gamma_{n+1}^{-1}\right]-\SS\left[\Delta_n\Gamma_{n}^{-1}\right].$$
  \end{thm}

\begin{proof}
	To use Theorem \ref{t007}, we need to calculate $\ti l_t(0)$. Once again we
	will switch to the Fourier transform and use the identity 
	$\ti l_t(0)=\int_{-t}^t \FF(\ti l_t)$. Recall that $\ti l_t=H^\mu k_t$.
	For $\mu\in M^+(\T)$, 
	$$H^\mu k_t=Kk_t\mu-k_tK\mu +ck_t,$$
	see Remark \ref{rem007}. The last term can be disregarded, see Theorem \ref{t1} and Remark \ref{remHk}.
	
	The first term $Kk_t\mu$ is equal to the Hilbert transform $Hf\mu$ on $\R$.
	Recall that $\FF(k_t\mu)=1$ on $(-t,t)$ and that $\FF (Hg)(s)={\rm sign}\ s\cdot \hat g$.
	Hence,
	$$\int_{-t}^t\FF(Kk_t\mu)(s)ds=\int_{-t}^t{\rm sign}\ s\cdot\FF(k_t\mu)ds=
	\int_{-t}^t{\rm sign}\ s\ ds=0.$$
	Therefore
	$$\ti l_t(0)=\int_{-t}^t \FF(k_tK\mu)(s)ds=\int_{-t}^t f_t\ast({\rm sign}\ s\cdot\hat\mu(s))ds,$$
where $f_t=\hat k_t$ like in the last proof. 

Suppose that $t\in(\frac n2,\frac{n+1}2)$. 
Let $I_k,\ k=1,2,...,2n+1$ be the intervals of the partition of $(-t,t)$ like
in the last proof. 
As was established there, $f_t$ is constant on the intervals $I_k$
and its values on the odd and even intervals are given by 
\eqref{eq1500} and \eqref{eq1501} respectively. It follows that
the values of $q=f_t\ast({\rm sign}\ s\cdot\hat\mu(s))$ on odd $I_k$ are given 
by 
$$\begin{pmatrix}
		q(\xi)\\q(\xi+1)\\ \vdots \\ q(\xi+n)
	\end{pmatrix}=\Delta_{n+1}\Gamma_{n+1}^{-1}
	\begin{pmatrix}
		1\\1\\\vdots \\1
	\end{pmatrix}$$
for $\xi\in I_1$ and for even $I_k$ by
$$\begin{pmatrix}
	q(\xi)\\q(\xi+1)\\ \vdots \\ q(\xi+(n-1))
\end{pmatrix}=\Delta_n\Gamma_{n}^{-1}
\begin{pmatrix}
	1\\1\\\vdots \\1
\end{pmatrix},$$
for $\xi\in I_2$. Finishing like in the last proof we obtain the formula in the statement.
\end{proof}

\subsection{Canonical systems and orthogonal polynomials on the unit circle}
\label{secOPUC}
In this section we study the connection between the inverse spectral problems 
for CS with periodic spectral measures and orthogonal polynomials 
on the unit circle. We show that the formula for $h_n$ from Theorem
\ref{tPer} can be interpreted as point evaluations of such polynomials.

 For $\mu\in M_+(\T)$ we
denote by $\ONP(\mu)$ the family of  polynomials of $z$ on the unit circle 
orthonormal in $L^2(\mu)$. By $\phi_n\in\ONP(\mu)$ we denote the polynomial of degree n.

\begin{thm}\label{tONP} Let $\mu\in M^+(\T)$ and let  $\{\phi_n\}$ be  ONP($\mu$). We have
	$$h^\mu_n=|\phi_n(1)|^2.$$\end{thm}

\no In the proof we will use the following notation. We denote
$${\mathcal P}_n(\mu)={\rm span}\{1,\dots, z^n\}\subset L^2(\mu).$$
Then
$$K^w_n\equiv K_n(z,w)=\sum_{j=0}^n\phi_j(z)\overline{\phi_k(w)}$$
is the reprokernel of ${\mathcal P}_n(\mu)$.

In this section $m=m_\T$ stands for the normalized Lebesgue measure on $\T$, $m({\mathbb T})=1$. For $\ONP(m)$ we have
$\phi_n(z)=z^n$ and $$\mathring K (z,w)=\sum_{j=0}^nz^j\bar w^j.$$
If $w=1$, then we'll write $k_n(z)$ for $K_n(z,1)$, in particular
\begin{equation}\mathring k_n(z)=\sum_{j=0}^n z^j.\label{eqbf1}\end{equation}

We consider truncated Toeplitz operators on $\PP_n=\PP_n(m)$:
$$T_\mu= T_\mu^n:~{\mathcal P}_n\to {\mathcal P}_n$$
defined by
$$ (T_\mu p,q)=\int p\bar q~d\mu\quad \forall\ p,  q\in\PP_n.$$
For absolutely continuous  measures $\mu$ with $L^2(m)$-densities $T_\mu:p\mapsto P_n(p \mu)$ where $P_n$ is the orthogonal projection onto ${\mathcal P}_n$ in $L^2(m)$. Note that the  matrix of $T_n$ in the basis $\{z^j\}$ is
$$T_{jk}=(Tz^k,z^j)=\int z^k\bar z^j~d\mu=\gamma_{j-k},$$
which is the matrix $\Gamma^{\scriptscriptstyle T}_{n+1}(\mu)$, the transpose of  $\Gamma_{n+1}(\mu)$ from Section \ref{secMom}.

\begin{lem} If $j_n:\PP_n\to \PP_n(\mu)$ denote the 'identity' embedding operators, then
	$$j_n^*j_n=T_\mu\ {\rm on }\ \PP_n.$$
\end{lem}
\begin{proof}
	$$(j_n^*j_n p,q)=(j_np,j_nq)_\mu=\int p\bar q~d\mu=(T_\mu p,q).$$\end{proof}

\begin{lem}  $$\forall w\in{\mathbb C}, \quad j_n\left[T^{-1}_\mu\mathring K_w \right]=K_w$$
	In particular,
	$$j_n\left[T^{-1}_\mu\mathring k_n\right]=k_n.$$\end{lem}
\begin{proof} Recall that
	$$\forall p, q\in \PP_n,\quad (jp,jq)_\mu=(T_\mu p, q).$$
	Let $p=j^{-1} K_w$. Then we have
	$$( \mathring K_w, q)=\bar q(w)=(K_w,jq)_\mu=(Tj^{-1} K_w,q).$$
\end{proof}

\begin{proof}[Proof of Theorem \ref{tONP}] We will first  show that
$$\Sigma\left(\Gamma^{-1}_{n+1}(\mu)\right)=k_n(1).$$
Indeed, for an $n\times n$ matrix $A$ and the standard basis  $e_0, \dots, e_n$ in $\R^n$, 
$$\Sigma(A)=\sum_j\sum _k (Ae_j, e_k)= (A{\bf 1},  {\bf 1}),\ {\rm where} \ {\bf 1}=\sum_{j=0}^ne_j.$$
Applying this to the matrix $\Gamma^T_{n+1}(\mu)$ of $T_\mu$ with respect to the
basis $1,...,z^n$ and taking into account that ${\bf 1}=\stackrel {\circ}k_n$, see \eqref{eqbf1}, we get
$$\Sigma\left(\Gamma^{-1}_{n+1}(\mu)\right)=\Sigma\left(\left[\Gamma^t_{n+1}(\mu)\right]^{-1}\right)=\left(T_\mu^{-1} \stackrel {\circ}k_n, \stackrel {\circ}k_n\right)=
\left(k_n, \stackrel {\circ}k_n\right)=k_n(1).
$$
\ss\no To finish the proof of the theorem we observe that
$$k_n(1)=K_n(1,1)=\sum_{j=0}^n\left|\phi_j(1)\right|^2,$$
and
$$h_n=\Sigma\left[\Gamma_{n+1}^{-1}\right]-\Sigma\left[\Gamma_{n}^{-1}\right]=k_n(1)-k_{n-1}(1)=|\phi_n(1)|^2.$$
\end{proof}

Theorem \ref{tONP} allows us to describe the change of $h_n$ corresponding to a shift of a periodic spectral measure. For $\eta\in\mathbb T$ we define  $\mu^{(\eta)}(\zeta)=\mu(\eta\zeta)$.

\begin{cor}$$ h_n^{(\eta)}=|\phi_n(\eta)|^2.$$\end{cor}
\begin{proof} ONP of $\mu^{(\eta)}$ are the polynomials $\phi_n(\eta\zeta)$.\end{proof}

  \bs\section{Examples of spectral problems}\label{secEx}

The goal of this section is to illustrate our formulas with explicit examples and calculations.

\bs\subsection{Constant Hamiltonians}\label{sec0001}

Let us first consider spectral problems, inverse and direct, for systems
with constant Hamiltonians.

\begin{lem} Let the constants $h_1>0, h_2>0 , g\in\R$ satisfy   $h_1h_2=1+g^2$, and
$$\HH=\begin{pmatrix} h_1&g\\g&h_2\end{pmatrix}\in{\rm SL}(2,\mathbb R).$$
Then
$$M(t,z)= \begin{pmatrix} C-gS&-h_2S\\h_1S&C+gS\end{pmatrix},$$where $C:=\cos zt$ and $S:=\sin zt$\end{lem}
\begin{proof} Just verify $M(0)=I$ and $\Omega\dot M=z\HH M$ using $\det(\HH)=1$. To guess the solution use undetermined coefficients.\end{proof}

\begin{cor} $\mu=h_1^{-1}$.\end{cor}
\begin{proof} 
	Notice that the solution can be rewritten as
	$$\begin{pmatrix} u_z(t)\\v_z(t)\end{pmatrix}
	\begin{pmatrix} C-gS\\h_1S\end{pmatrix}=\begin{pmatrix} 1&-g\\0&h_1\end{pmatrix}\begin{pmatrix} C\\S\end{pmatrix}=\begin{pmatrix} \frac1{\sqrt{h_1}}&-\frac g{\sqrt{h_1}}\\0&\sqrt{h_1}\end{pmatrix}\begin{pmatrix} \sqrt{h_1}C\\\sqrt{h_1}S\end{pmatrix}$$
Since the determinant of the last $2\times2$ matrix is 1, by Theorem 2.2 from
\cite{Rem}
$F_t(z)=\sqrt{h_1}C-i\sqrt{h_1}S$ generates the same   chain
of dB spaces as $E_t(z)=u_z(t)-iv_z(t)$. But
$$\BB(F_t)=\BB\left(\sqrt{h_1}e^{-izt}\right),$$
and $h_1^{-1}=1/|F_t|^2$ is the spectral measure. \end{proof}

 Alternatively, to prove the last corollary we can apply the general algorithm described in Section 
 \ref{secHmu} to the even measure  $\mu=a, \ a=h_1^{-1}$. We have $\Gamma(\mu)=a I$, so $\Gamma_n^{-1}=\frac1aI_n$ and
$S\left( \Gamma_n^{-1}\right)=\frac na$. Then $h_n=\frac {n+1}a-\frac na=\frac 1a$.

To practice calculating $h_2$ using the dual measure, notice that for $\mu=a$, $i\KK\mu+ib=a+ib$ and  $\PP\ti\mu_b=\ti a:=\frac a{a^2+b^2}$. We have
$$h=h^\mu=\frac1a,\quad g^\mu=0,\quad \ti h=h^{\ti\mu}=\frac1{\ti a},\quad g^{\ti\mu}=0.$$
Hamiltonians corresponding to the pair of measures $\mu$, $\ti\mu_b$ are given by the constant matrix
$$H=\begin{pmatrix}h&kh\\kh&\ti h\end{pmatrix},\ k\in\R.$$

\bs\subsection{Spectral measure $\mu=1+\cos x$}\label{sec1+c}
 Let us solve the inverse spectral problem for this particular  measure. 
 Recall that by Theorem \ref{tPer}, for a $2\pi$-periodic even measure  $\mu=1+\cos x$ the Hamiltonian
 is diagonal and constant on 
 $$ \left(0,\frac12\right),\; \left(\frac12,1\right),\; \left(1,\frac32\right),\;\dots$$

To find the values of the Hamiltonian on these intervals using Theorem \ref{tPer} notice that
$$\Gamma(\mu)=\begin{pmatrix}1&\frac12&0&\dots\\\frac12&1&\frac12&\dots
\\ 0 &\frac 12 & 1& \dots\\  \vdots & \vdots & \vdots &\ddots \end{pmatrix},$$
 $$\Gamma_1^{-1}=(1),\ \  \Gamma_2^{-1}=\frac43\begin{pmatrix}1&-\frac12\\-\frac12&1\end{pmatrix}, ~\dots$$ 
and $$\SS_1=1,\ \SS_2=\frac43,~ \dots$$ for $\SS_n=\SS[\Gamma_n^{-1}]$.

  We conclude that the values of $h_0,h_1,h_2...$ are 
 $$1,\  \frac13,\   \frac23,\   \frac25,\   \frac 35,\   \frac 37,\   \frac 47,\  \frac 49,...$$
 $$...,h_{2n}=\frac{n+1}{2n+1},
 h_{2n+1}=\frac{n+1}{2n+3},...$$

 Since  the  Hamiltonian $\HH$ is diagonal and $\det\HH=1$,
 the the values for the second diagonal term can be calculated as $1/h_n$.

From \eqref{eqEven} we can find the even dual measure $\tilde\mu$:
$${\mathcal P}\mu=\Re(1+e^{iz}),\quad {\mathcal P}\tilde\mu=\Re\frac1{1+e^{iz}}
=\frac12+\frac12\Re\frac{1-S}{1+S},$$
so
$$\tilde\mu=\frac12+\pi\sum_{-\infty}^\infty\delta_{\pi+2\pi k},\quad \Gamma(\tilde\mu)=\begin{pmatrix}1&-\frac12&\frac12&-\frac12& \dots\\-\frac12&1&-\frac12&\frac12&\dots\\
\frac 12&  -\frac 12 &1 &-\frac 12& \dots\\ 
\vdots &\vdots & \vdots & \vdots &\ddots\end{pmatrix}$$
 Computation gives $\tilde h_0=1$, $\tilde h_1=3$, $\tilde h_2=\frac32, \tilde h_3=\frac52$, etc., as
 expected from the above calculation for $h_k$.
 
 Note that one can find $h_n$ for $\mu=1+\cos x$ and many other periodic measures using tables of orthogonal polynomials on the unit circle available in the literature,
 together with Theorem \ref{tONP}. To find $\ti h_n$ one can apply Theorem
 \ref{tONP} to the dual measure, see Section \ref{secDM}. Further examples
 based on that approach will be given below, see Sections \ref{secEP} and \ref{secED}.

\subsection{Four steps by hand} Suppose $\mu$ is even (then the moments $\gamma_k$ are real). As before, denote $\SS_n=\SS[\Gamma_n^{-1}]$.
\begin{lem} $$\SS_1=\frac 1{\gamma_0},\quad \SS_2=\frac 2{\gamma_0+\gamma_1},\quad \SS_3=\frac{3\gamma_0-4\gamma_1+\gamma_2}{(\gamma_0+\gamma_2)\gamma_0-2\gamma_1^2},$$$$ \SS_4=\frac{2(2\gamma_0-\gamma_1-2\gamma_2+\gamma_3)}{(\gamma_0+\gamma_3)(\gamma_0+\gamma_1)-(\gamma_1+\gamma_2)^2}.$$ \end{lem}
\begin{proof} The first formula is obvious. The second follows from
$$\Gamma_2^{-1}=\frac1{\gamma_0^2-\gamma_1^2}\begin{pmatrix}\gamma_0&-\gamma_1\\-\gamma_1&\gamma_0\end{pmatrix}$$
Third formula: let $\vec{x}$ be defined by
$$ \begin{pmatrix}\gamma_0&\gamma_1&\gamma_2\\\gamma_1&\gamma_0&\gamma_1\\\gamma_2&\gamma_1&\gamma_0\end{pmatrix}
\begin{pmatrix}x_1\\x_2\\x_3\end{pmatrix}=\begin{pmatrix}1\\1\\1\end{pmatrix}$$
By symmetry we have $x_1=x_3$, and
$$\SS_3=(\Gamma_3^{-1}\vec{1},\vec{1})=x_1+x_2+x_3=2x_1+x_2.$$
The variables $2x_1$ and $x_2$ satisfy the system
$$\frac{\gamma_0+\gamma_2}2(2x_1)+\gamma_1x_2=1,\qquad \gamma_1(2x_1)+\gamma_0x_2=1$$
(the first two equations in the $3\times 3$ system above), and we have
$$\SS_3=2x_1+x_2=\SS(M^{-1}),$$
where $M$ is the 2$\times$2 matrix of the last system. The derivation of the forth formula is as simple -- we use the symmetry $x_1=x_4$ and $x_2=x_3$, so again we only need to invert a 2$\times$2 matrix.  \end{proof}

\begin{example} If $\gamma_0=1, \;\gamma_1=\frac12,\;\gamma_{\ge2}=0$, then we have
$$\SS_1=1, \quad \SS_2=\frac43,\quad \SS_3=2,\quad \SS_4=\frac{12}5,$$
and if $\gamma_0=1, \;\gamma_1=-\frac12,\;\gamma_2=\frac12$, then
$$\SS_1=1, \quad \SS_2=4,\quad \SS_3=\frac{11}2.$$
\end{example}

\bs\subsection{Finding dual measures}\label{secDM} Recall that if $\phi$ is a Schur function, then we define Clark's measures $\sigma_\alpha^\phi$ on $\hat\R$ for $\alpha\in\T$ via the equation
$$\PP\sigma_\alpha=\Re\frac{\alpha+\phi}{\alpha-\phi},$$
see Section \ref{secClark}.

 Let us find all dual measures  for $\mu=1+\cos x$ using the formulas from
 Section \ref{secCD}. 
 It is clear the the corresponding Schur function $\phi$ is a periodic 
 non-constant function and 
 therefore $\phi-\a\not\in L^2(\R)$ for any $\a$. Hence the dual measures
 have no masses at $\infty$.
 
 For $S(z):=e^{iz}$ we have 
 $$\PP\mu=\Im\KK\mu=\Re(1+S) \ {\rm and}\ \KK\mu=i(1+S).$$ 
 Hence by \eqref{eqDM},
  \begin{equation}\PP\tilde{\mu}_b=\Re\frac1{1+S-ib}.\label{eq400}\end{equation}

\ss\no  Case $b=0$ (the only choice such that $\tilde{\mu}$ is even):
$$\PP\ti\mu=\frac12+\frac12\Re\frac{1-S}{1+S}=\PP\left[\frac12+\frac12\sigma^S_{-1}\right],$$
and by \eqref{eqClarkpp},
$$\ti\mu=\frac12+\pi\sum_{n\in\Z}\delta_{(2n+1)\pi}.$$
Let us also find the moments of $\tilde{\mu}$:
$$\gamma_k=\frac1{4\pi}\int_0^{2\pi}e^{-ikx}dx+\frac12\sum_{n=-\infty}^\infty\int_0^{2\pi}e^{-ikx}\delta_{(2n+1)\pi}(x)=I+II.$$
In $II$, only the term with $n=0$ has an atom  in $(0,2\pi)$, so
$$II=\frac12 e^{-ik\pi}=\frac12(-1)^k.$$
Of course, $I=\frac12$ if $k=0$ and $I=0$ otherwise.

  Case $b\ne0$. From \eqref{eq400} one can conclude that
  $\PP\ti\mu_b$ is bounded near $\R$ and therefore $\ti\mu$ is absolutely
  continuous. Hence, on $\R$,
$$\ti\mu=\Re\frac1{1+S-ib}=\Re\frac1{1+\cos x+i(\sin x -b)}=\frac{1+\cos x}{2+2\cos x-2b\sin x+b^2}.$$
(Note that formally setting $b=0$ in the last formula we get the wrong answer because $\ti\mu_0$ is not a.c.) To find the moments, we may use geometric progression,
$$\ti\mu=\Re\left[\frac1{1-ib}~
\frac1{1+\frac{e^{ix}}{1-ib}}\right]=\Re\left[\frac1{1-ib}-\frac{e^{ix}}{(1-ib)^2}+\frac{e^{2ix}}{(1-ib)^3}-\dots\right],$$
and we get
$$\gamma_0=\frac1{1+b^2},\quad \gamma_1=-\frac12\frac1{(1-ib)^2}, \quad \gamma_2=\frac12\frac1{(1-ib)^3}, \quad \dots$$
(and $\gamma_{-k}=\overline{\gamma_k}$). This follows from the representations like
$$\Re\frac{e^{ix}}{(1-ib)^2}=\frac12\frac{e^{ix}}{(1-ib)^2}+\frac12\frac{e^{-ix}}{(1+ib)^2}.$$
Note that the formal limit $b=0$ gives the right moments.

Using the moments and Theorem \ref{tPer} one can recover the Hamiltonians
(in the case $b=0$ it was done in Section \ref{sec1+c}).

Generalizing the example from Section \ref{sec1+c},  let us consider $\mu=1+a\cos$ where $|a|<1$. Then $\PP\mu=\Re(1+aS)$, and
$$\PP\ti\mu=\Re\frac1{1+aS}$$
for a symmetric case ($b=0$). Thus
$$\ti\mu=\frac{1+a\cos x}{1+2a\cos x+a^2},$$
and
$$\Gamma(\ti\mu)=\left \{1, ~ -\frac a2, ~\frac{a^2}2,~\frac{a^3}2, ~\dots\right\}.$$
(Here we use the notation $M=\{a_1,a_2,...\}$ for a symmetric Toeplitz matrix $M$ with
the first row $a_1,a_2,...$.)

This follows from the representation
$\ti\mu=\Re[1-aS+a^2S^2-\dots]$. Note that we get the right moments in the limit $a=1$.  Of course,
$$\Gamma(\mu)=\left \{1, ~ \frac a2, ~0,~0, ~\dots\right\}.$$


\bs\subsection{The Poisson measure}\label{secEP} Fix $a\in\D$ and consider the
$\mu\in M^+(\T),\ \mu=P_a$, where $P_a$ is the Poisson kernel on $\T$,
$$P_a(\xi)=\frac{1-|a|^2}{|\xi-a|^2}.$$
Recall that a polynomial is called monic if its main coefficient is equal to 1.

\begin{lem}[Szego, Bernstein] The monic orthogonal polynomials 
	on $\T$ with respect to $P_a$ are
	$$\Phi_0(z)=1, \qquad \Phi_n(z)=z^n-az^{n-1},\quad (n\ge1),$$
	are orthogonal with respect to $\mu$. Their norms are
	$$\|\Phi_0\|_\mu=1,\qquad \|\Phi_n\|^2_\mu=1-|a|^2.$$\end{lem}
\begin{proof} If $n>k\ge 1$, then
	\begin{align*} (\Phi_n, \Phi_k)_\mu&=\int_{\T}(z^n-az^{n-1})(z^{-k}-\bar az^{1-k})P_a(z)\\&=a^{n-k}+|a|^2a^{n-k}-a a^{n-k-1}-\bar aa^{n-k+1}=0.\end{align*}
	Also, $$(\Phi_n,\Phi_0)_\mu=\int_\T(z^n-az^{n-1})P_a=a^n-aa^{n-1}=0.$$
	Finally,
	$$\|\Phi_n\|^2=\int_\T(1+|a|^2-a\bar z-\bar a z)P_a=1+|a|^2-|a|^2-|a|^2.$$
\end{proof}

For ONP we have
$$\phi_0(z)=1,\qquad \phi_n(z)=\frac{z^{n-1}(z-a)}{\sqrt{1-|a|^2}}\quad (n\ge1).
$$

\begin{cor} $$h_0=1,\qquad h_n=\frac{|1-a|^2}{1-|a|^2} \quad (n\ge1).$$\end{cor}

{\bf Remarks.} (a) The periodic measure $\mu^\R(x)=\mu^\T(e^{ix})$ is even if $\mu^\T(z)=\mu^\T(\bar z)$.
This is the case when $a$ is real; then
$$h_n=\frac{1-a}{1+a} \qquad (n\ge1).$$

\ss\no (b) Shift of the spectral measure, $\mu(\xi)\mapsto \mu_\eta(\xi)=\mu(\eta\xi),\ \eta\in\T$,  results in the Hamiltonian with
$$h_n^{(\eta)}=\frac{|\eta-a|^2}{1-|a|^2}\qquad (n\ge1).$$
If $a$ is real, then
$$h^{(-1)}_n=\frac{1+a}{1-a}=h_n^{-1}$$
which suggests that the measures $P_a$ and $P_{-a}$ are dual. This relation 
can also be established directly for all $a\in\T$:

\begin{lem} For any $a\in\D$, the measures $P_a^\R$ and $P_{-a}^\R$ are dual.
\end{lem}
\begin{proof}  The harmonic function
	$$\Re\frac{1+\bar a S}{1-\bar a S}=\frac{1-|a|^2|S|^2}{|1-\bar aS|^2}$$
	is positive in the upper halfplane and its boundary values are $P_a^\R(x)$. One of the dual  measures can then be found from
	$$\Re\frac{1-\bar a S}{1+\bar a S}=P_{-a}\qquad ({\rm on}\; \R).$$
\end{proof}


\begin{cor} We have $$
	h_n\tilde h_n=\frac{|1-a^2|^2}{(1-|a|^2)^2}\qquad (n\ge1)$$
	and therefore
	$$g_n^2=h_n\tilde h_n-1=\frac {4(\Im a)^2}{(1-|a|^2)^2}.$$
\end{cor}

To illustrate our formulas further
let us use the above calculations and obtain the measure $P_a$ in the case of real $a$ solving
the direct spectral problem. In this case the Hamiltonian is diagonal and, as was shown above,  equal
to $I$ on $\left(0,\frac 12\right)$ and to
$$\begin{pmatrix}
	h & 0 \\ 0 &h^{-1}
\end{pmatrix}
$$
on the rest of $\R_+$, where 
$$h=\frac{1-a}{1+a}.$$

Using the formulas for constant Hamiltonians from Section  \ref{sec0001}, we obtain that
the matrizant is equal to 
$$\begin{pmatrix}
	\cos (z/2) & -\sin(z/2) \\
	\sin(z/2) & \cos(z/2)
\end{pmatrix}$$
on $\left(0,\frac 12\right)$ and to 
$$\begin{pmatrix}
	\cos zt &-h^{-1}\sin zt \\
	h \sin zt& \cos zt
\end{pmatrix}\begin{pmatrix}
	\cos (z/2) & -\sin(z/2) \\
	\sin(z/2) & \cos(z/2)
\end{pmatrix}$$
at "time" $\frac12+t$.

The corresponding (normalized) HB function is
$$E=A-iC$$
where
$$A=\cos zt\cos(z/2)-h^{-1}\sin zt\sin(z/2), \quad C=h\sin zt\cos(z/2)-\cos zt\sin(z/2).$$
The function $E$ has the same dB space as the function
$$ E^*=\sqrt hA-\frac i{\sqrt h}C.$$
A small miracle happens in the following computation:
$$| E^*|^2=h|A|^2+\frac1h|C|^2=h\cos^2\frac z2+\frac1h\sin^2\frac z2,$$
with the final expression showing no dependence on $t$. It follows that the spectral measure is
$$\mu(x)=\frac 1{|E^*(x)|^2}=\frac1{h\cos^2\frac x2+\frac1h\sin^2\frac x2}.$$
This is equal to
$$P_a(x)=\frac{1-a^2}{(\cos x-a)^2+\sin ^2x}.$$

\bs\subsection{Delta measure plus a constant}\label{secED} 
Let now $\gamma\in (0,1)$ be a parameter, and consider $\mu\in M^+(\T)$ defined as
$$\mu^\T=(1-\gamma)+2\pi\gamma\delta_1.$$
As before, we identify $\mu$ with a periodic measure on $\R$,
$$\mu= (1-\gamma) + 2\pi\gamma\sum_n \d_{2\pi n}.$$

\begin{lem}[\cite{Simon}] The monic orthogonal  polynomials on $\T$ with respect to $\mu$
	are 
	$$\Phi_0(z)=1,\qquad \Phi_n(z)=z^n-\alpha_{n-1}(1+\dots+z^{n-1}),\quad (n\ge1),$$
	where
	$$\alpha_n=\frac\gamma{1+n\gamma}.$$
 The norms are
	$$\|\Phi_n\|^2=1-n\alpha_{n-1}\gamma,\quad (n\ge0).$$
	\end{lem}


Since the measure is even, the corresponding Hamiltonian is diagonal
with the diagonal entries on $\left(n,n+\frac 12\right)$ given by Theorem \ref{tONP}:

\begin{cor} $$h_0=1,\qquad h_n=\frac{|\Phi_n(1)|^2}{\|\Phi_n\|^2}=\frac{(1-n\alpha_{n-1})^2}{1-n\alpha_{n-1}\gamma},$$
	$$\ti h_0=1,\qquad \ti h_n=\frac{1}{h_n}=\frac{1-n\alpha_{n-1}\gamma}{(1-n\alpha_{n-1})^2}.$$
\end{cor}

%
%
%
%


\bs\subsection{Spectral measure $\mu=\alpha+\beta\pi\delta_0,\ \a>0,\b\geq 0$}\label{secdelta0}
Let us now consider this example of a non-periodic even spectral measure. Using
the formulas
$$\hat\delta=\frac1{\sqrt{2\pi}}\cdot 1,\qquad \hat1=\sqrt{2\pi}\cdot\delta,$$
we obtain
$$\hat\mu=\alpha\sqrt{2\pi}\cdot\delta+\frac{\beta\pi}{\sqrt{2\pi}}\cdot 1.$$
Next, we use Theorem \ref{tTT} to find the Fourier transform $\psi_t$ of
the reproducing kernel.
Since $\psi\ast 1=\int\psi$, Theorem \ref{tTT} gives
$$\forall x\in(-t,t),\qquad \alpha\sqrt{2\pi}\psi_t(x)+\frac{\beta\pi}{\sqrt{2\pi}}\int_{-t}^t\psi_t=1.$$
Thus
$$\psi_t=c(t)1_{(-t,t)},\qquad \int_{-t}^t\psi_t=2tc(t),$$
and
$$\alpha\sqrt{2\pi}c(t)+\frac{\beta\pi}{\sqrt{2\pi}}2tc(t)=1,$$
so
$$c(t)=\frac1{\alpha\sqrt{2\pi}+\frac{\beta\pi}{\sqrt{2\pi}}2t}=\frac{\sqrt{2\pi}}{2\pi\alpha+2\pi t\beta}$$
and
$$k_t(0)=\frac1{\sqrt{2\pi}}\int\psi_t=\frac{2tc(t)}{\sqrt{2\pi}}=\frac1\pi\frac t{\alpha+t\beta}.$$
It follows that
$$h^\mu(t)=\pi\frac d{dt} k_t(0)=\frac d{dt}\frac t{\alpha+t\beta}=\frac\alpha{(\alpha+t\beta)^2}$$
and that  $$H(t)=\begin{pmatrix}  \frac\alpha{(\alpha+t\beta)^2}&0\\0&\frac{(\alpha+t\beta)^2} \alpha     \end{pmatrix}$$
is a  unique diagonal Hamiltonian with spectral measure $\mu$.


$$$$

\begin{remarks}
	
\begin{itemize}
	$$$$

\item Note that in the limiting case $\alpha\to 0$ (and $\beta=1$):
$$\mu(x)\to \pi\delta(x),\qquad h^\mu(t)\to \delta(t).$$
Indeed,
$$\int_0^\infty\frac\alpha{(\alpha+t)^2}~dt=-\frac\alpha{\alpha+t}\Big|_{t=0}^\infty=1$$
This shows how using our methods developed for PW-measures one can solve the  inverse problem for a non-PW measure $\mu=\pi\delta$

\item The dual spectral measure of $H$ is
$$\tilde\mu(x)=\frac{\alpha x^2}{\alpha^2 x^2+\beta^2}.$$

\begin{proof}
$$\PP\mu(z)=\alpha+\beta\frac y{x^2+y^2}=\Re\left[\alpha+\frac{i\beta}z\right]$$
We have
$$\frac1{\alpha+\frac{i\beta}z}=\frac z{\alpha z+i\beta};$$
 on $\R$ the real part is
 $\frac{\alpha x^2}{\alpha^2 x^2+\beta^2}$.\end{proof}

 \item It is a good exercise to show directly that if $\ti\mu=x^2/(1+x^2)$, then $h^{\ti\mu}=(1+t)^2$.  Note that 
 $$\FF\left(\frac1{1+x^2}\right)(\xi)=\sqrt{\frac \pi 2}e^{-|\xi|},$$
 and we need to solve equations like
 $$\psi+ce^{-|\xi|}\ast\psi=1$$
 (this convolution is called the Beurling transform of $\psi$ in \cite{HJ}). Hint:
 $$\int\psi_t(y)e^{-|x-y|}dy=\int_{-t}^x+\int_x^t.$$
 
 \end{itemize}
 
\end{remarks}

 \bs\subsection{Adding an eigenvalue at the origin} Here is a generalization of the last example studied with different methods by H. Winkler in \cite{W}.
 
 \begin{thm} Let $\mu$ be a PW-sampling measure. For $r\geq 0$, let  $\mu_r=\mu+r\pi\delta$. Then
 $$h^{\mu_r}(t)=\frac{h^\mu(t)}{\left(1+r\int_0^th^\mu\right)^2}$$\end{thm}

 \begin{example} $\mu=\alpha m$, $h^\mu=\alpha^{-1}$, and
 $$h^{\mu_r}(t)=\frac{\alpha^{-1}}{\left(1+\alpha^{-1}rt\right)^2}=\frac{\alpha}{(\alpha+rt)^2}$$
 \end{example}

 \bs\begin{proof} Fix $r>0$. Let $\psi_t$ and $f_t$ be solutions of the equations from Theorem \ref{tTT},
 $$\psi_t\ast\hat\mu_r=1,\qquad f_t\ast \hat\mu_r=1\qquad{\rm on}\quad (-t,t),$$
 so that $$h^\mu=\sqrt{\frac\pi2}\frac d{dt}\int\psi_t$$
 and
 \begin{equation}
 h^{\mu_r}=\sqrt{\frac\pi2}\frac d{dt}\int f_t.\label{eq500}\end{equation}
 From $f\ast \hat\mu_r=1$ we find
 $$f_t\ast\hat\mu+r\sqrt{\frac\pi2}\int f_t=1,$$
 or
 $$f_t\ast\hat\mu=1-c(t),\qquad c(t)=r\sqrt{\frac\pi2}\int f_t.$$
 It follows that
 $$f_t=(1-c(t))\psi_t,$$
 where $c(t)$ satisfies the equation
 $$c(t)=r\sqrt{\frac\pi2}(1-c(t))\int \psi_t.$$
 
 Hence, $c(t)=a/(1+a)$ where
 $a(t)=r\sqrt{\pi/2}\int\psi_t$ and
 $$ \frac d{dt} a(t)=rh^\mu(t),\ a(0)=0.$$
From \eqref{eq500} we have 
$$h^{\mu_r}(t)=\frac d{dt} \frac{c(t)}r,$$
which yields the statement.\end{proof}

\begin{remark} 
	Note that $\mu_{x+y}$ can be obtained  as  $\mu_{x+y}=\mu_x+y\d_0$. For $h_x=h^{\mu_x}$ the last theorem yields the equation
	$$h_{x+y}(t)=\frac{h_x(t)}{\left(1+y\int_0^th_x\right)^2}$$
 with the initial condition $h_0=h^\mu$.

 $\triangle$
\end{remark}

 \bs\subsection{Spectral measure $\mu=\alpha+\beta\pi\delta_\lambda$} 
 Here $\alpha>0$, $\beta>0$, and $\lambda\in \R$ are  parameters. 
 Unlike Section \ref{secdelta0}, if $\l\neq 0$, then the measure is not even. 
 
 (i) We want to find
 $$h^\mu=\sqrt{\frac\pi2}\frac d{dt}\int\psi_t, $$
 where
 $$\hat\mu\ast \psi_t=1\qquad{\rm on}\quad (-t,t).$$
 We will use the notation $e_\lambda(x)=e^{i\lambda x}$. 
 We have
 $$\hat\mu=\alpha\sqrt{2\pi}\delta_0+\frac{\pi\beta}{\sqrt{2\pi}}e_{-\lambda}$$
 Note that
 $$(e_{-\lambda}\ast\psi)(x)=e_{-\lambda}(x)\int e_{\lambda}\psi,
$$
 so
 $$\psi_t=\left[\frac1{\alpha\sqrt{2\pi}}-\frac\beta{2\alpha}c(t)e_{-\lambda}\right]\cdot1_{(-t,t)},$$
 where
 $$c(t)=\int e_\lambda\psi_t.$$
 Also,
 $$\int_{-t}^te_\lambda=\frac{2\sin\lambda t}\lambda.$$
 From the last 3 equations we get
 $$c(t)=\sqrt{\frac2\pi}\frac1{\alpha+\beta t}\frac{\sin \lambda t}\lambda,$$
 and
 $$\sqrt{\frac\pi2}\int\psi_t=\frac t\alpha-\frac\beta\alpha\frac1{\alpha+\beta t}\left(\frac{\sin \lambda t}\lambda\right)^2.$$
Hence,
 $$h^\mu(t)=\frac d{dt}\left[\frac t\alpha-\frac\beta\alpha\frac1{\alpha+\beta t}\left(\frac{\sin \lambda t}\lambda\right)^2\right].$$
 E.g., if $\alpha=\beta=1$, then
 $$h^\mu(t)=\sin^2\lambda t+\left[\cos \lambda t-\frac{\sin \lambda t}{\lambda(1+t)}\right]^2$$

(ii) We now turn to the computation of $g^\mu$ for
$$\mu=\alpha+\pi\beta\delta_\lambda$$
To compute $g^\mu
$ we need
$$l_t(0)=\frac 1{2\pi i}\int_{-t}^t\psi_t*(\sigma \hat\mu)=\frac{\beta}{2i\sqrt{ 2\pi}}\int_{-t}^t\psi_t*(\sigma e_{-\lambda})=
$$
$$=\frac{\beta}{4i\pi\alpha}\int_{-t}^t 1_{[-t,t]}*(\sigma e_{-\lambda})-\frac{\beta^2 c(t)}{4\alpha i\sqrt{2\pi}}\int_{-t}^t(1_{[-t,t]}e_{-\lambda})*(\sigma e_{-\lambda}).
$$

The following formula can be established via direct calculations.

\begin{lem}
	$$\int_{-t}^t(1_{[-t,t]}e_{\lambda})*(\sigma e_{-\mu})=\frac 4{i(\lambda-\mu)}\left[\si_t(\lambda)-\cos(\lambda-\mu)t~\si_t(\mu)\right].$$	
\end{lem}

\ss\no Using the lemma to continue the calculation of $l_t(0)$,
$$l_t(0)=-\frac{\beta}{\alpha\lambda\pi}\left[t-\cos\lambda t\ \si_t(\lambda)\right]-\frac{\beta^2 c(t)}{\alpha\lambda \sqrt{2\pi}}\left[t\cos\lambda t-\si_t(\lambda)\right]=
$$$$
-\frac{\beta}{\alpha\lambda\pi}\left[t-\cos\lambda t\ \si_t(\lambda)\right]-\sqrt{\frac 2\pi}\frac{ \si_t(\lambda)}{\alpha + \beta t}\frac{\beta^2 }{\alpha\lambda \sqrt{2\pi}}\left[t\cos\lambda t-\si_t(\lambda)\right]=
$$$$
-\frac \beta{\alpha\pi\lambda}\left[t- \frac{\cos\lambda t \sin\lambda t}\lambda + \frac {\beta
	t}{\alpha + \beta t} \frac{\cos\lambda t \sin\lambda t}\lambda-        \frac \beta{\alpha + \beta t}\frac{\sin^2\lambda t}{\lambda^2}\right]=
$$$$-\frac \beta{\alpha\pi\lambda}\left[t- \frac {\alpha}{\alpha + \beta t} \frac{\cos\lambda t \sin\lambda t}\lambda-        \frac \beta{\alpha + \beta t}\frac{\sin^2\lambda t}{\lambda^2}\right]
$$
Now,
$$g^\mu(t)=\frac d{dt}l_t(0)=$$$$
-\frac \beta{\alpha\pi\lambda}\left[1  - \frac {\alpha}{\alpha + \beta t} \cos 2\lambda t+
\frac {\alpha\beta}{(\alpha + \beta t)^2}\frac{\cos\lambda t \sin\lambda t}\lambda 
-\frac \beta{\alpha + \beta t}\frac{\sin 2\lambda t}{\lambda} + \frac {\beta^2}{(\alpha + \beta t)^2}\frac{\sin^2\lambda t}{\lambda^2}\right]=
$$

$$
-\frac \beta{\alpha\pi\lambda}\left[1  - \frac {\alpha}{\alpha + \beta t} \cos 2\lambda t-
\frac {(\alpha\beta+2\beta^2 t)}{(\alpha + \beta t)^2}\frac{\sin 2\lambda t}{2\lambda}
+ \frac {\beta^2}{(\alpha + \beta t)^2}\frac{\sin^2\lambda t}{\lambda^2}\right].
$$

Let us simplify the above formula  in the particular case $\alpha=\beta=\lambda=1$:

$$
g^\mu(t)=-\frac 1{\pi}\left[1  - \frac {\cos 2 t}{1 +  t} -
\frac {(1+2 t)}{(1 +  t)^2}\frac{\sin 2 t}{2}
+ \frac {\sin^2 t}{(1 +  t)^2}\right]=
$$$$
\frac{(1+t)2\cos 2t+(1+2t)\sin 2t-(1-\cos 2t) }{2\pi(1+t)^2}-\frac 1\pi=
$$$$
\frac{2t(\cos 2t+\sin 2t)+3\cos2t+\sin 2t-1 }{2\pi(1+t)^2}-\frac 1\pi=
$$$$
\frac{(2t+1)(\cos 2t+\sin 2t)+2\cos2 t-1 }{2\pi(1+t)^2}-\frac 1\pi
$$

%

\subsection{Adding multiple solitons}\label{sec00011} Let us solve the inverse problem for
$$\mu=\alpha +\sum_n \pi\beta_n \delta_{\lambda_n}.$$
For its Fourier transform
we have
$$\hat\mu=\sqrt{2\pi}\alpha\delta_0+\sum_n\frac{\pi\beta_n}{\sqrt{2\pi}}e_{-\lambda_n},
$$
where $e_{s}(t)=e^{ist}$.

As before, let $\psi=\psi_t$ be the Fourier transform of the reproducing kernel $k_t(\cdot)=K_t(0,\cdot)\in \BB_t$. 
Theorem \ref{tTT} then gives
$$\hat\mu*\psi=1\textrm{ on }(-t,t)$$
which can be rewritten as
\begin{equation}\sqrt{2\pi}\alpha\psi(x) + \sum_n c_n(t)\frac {\pi\beta_n}{\sqrt{2\pi}}e_{-\lambda_n}(x)=1\label{e01}\end{equation}
where
$$c_n(t)=\int_\R e_{\lambda_n}\psi =\sqrt{2\pi}\hat\psi(-\lambda_n).$$
Solving \eqref{e01} for $\psi$ we get
\begin{equation}
	\psi(x)=\frac 1{\sqrt{2\pi}\alpha}-\sum_n\frac{\beta_n}{2\alpha}c_n(t)e_{-\lambda_n}(x).\label{e02}\end{equation}
Let us recall the notation for the sinc  function:
$$\si_t(x)=\frac{\sin (tx)}{ x}.$$
The reprokernel of $PW_t$ is $\mathring K_t(w,z)=\frac 1\pi\si_t(z-\bar w)$.

Direct calculations give

\begin{lem}
	$$	\int_{-t}^t e_\lambda=\frac{2\sin \lambda t}\lambda=2\si_t(\lambda)$$
	$$\int_{-t}^t e_\lambda e_{-\mu}=\frac{2\sin (\lambda-\mu) t}{\lambda-\mu}=2\si_t(\lambda-\mu).$$
	
\end{lem}

Using \eqref{e02},
$$c_n(t)=\int_{-t}^t e_{\lambda_n}\psi_t = \int e_{\lambda_n}\left[\frac 1{\alpha\sqrt{2\pi}}-\sum_k\frac{\beta_k}{2\alpha}c_k(t)e_{-\lambda_k}\right]=$$
$$
\frac 2{\alpha\sqrt{2\pi}}\si_t(\lambda_n)-t\frac{\beta_n}{\alpha}c_n(t)-\sum_{k\neq n}\frac{\beta_k}\alpha c_k(t)\si_t(\lambda_n-\lambda_k).$$
The last equation can be rewritten as
\begin{equation}\left(\alpha +\beta_n t\right) c_n +\sum_{k\neq n} \ \si_t(\lambda_n-\lambda_k)\beta_kc_k=\frac 2{\sqrt{2\pi}}\ \si_t(\lambda_n).\label{e03}\end{equation}
Let us introduce the following two matrices,
$$B=\left(\begin{array}{cccccc}
	\beta_1&  0 & \dots  & 0 & 0 & 0\\
	0&  \beta_2 & \dots & 0&  0 &  0\\
	\vdots& \vdots &  \ddots & \vdots  &  \vdots& \vdots \\
	0&  0& \dots & \beta_{n-2} & 0 & 0 \\
	0&  0& \dots & 0& \beta_{n-1}& 0\\
	0 & 0 & \dots & 0 & 0 & \beta_n
\end{array}\right),
$$
and
$$
S_t=\left(\si_t(\lambda_j-\lambda_k)\right)_{1\leq j,k\leq n},$$
and two vectors
$$
L_t=\sqrt{\frac 2\pi}\left(\begin{array}{c}
	\si_t(\lambda_1)\\
	\si_t(\lambda_2)\\
	\vdots \\
	\si_t(\lambda_n)
\end{array}\right),\ {\rm and}\  C(t)=\left(\begin{array}{c}
	c_1(t)\\
	c_2(t)\\
	\vdots \\
	c_n(t)
\end{array}\right).
$$

Our calculations above lead to the following

\begin{thm}\label{t01}
	$$\left(\alpha+S_t B\right)^{-1} L_t=C(t).$$
\end{thm}

It follows that $h_{11}=h^\mu$ can now be calculated via the following 
formula, which concludes the solution in the case when $\mu$ is even.

\begin{cor}
	$$	h^\mu(t)=\frac 1{\pi \alpha}-\frac 1{\sqrt{2\pi}\alpha}\frac d{dt}<B(\alpha+S_t B)^{-1} L_t,L_t>.$$
\end{cor}

In the general case, the off-diagonal term of $\HH$, $g^\mu$, can be 
calculated using the formulas of Section \ref{secGmu}. These calculations
will be presented elsewhere.

\subsection{Dual solitons}
Let us, once more, consider a measure $\nu=\alpha+\beta\pi\delta_0$. The dual measure $\mu=\ti\nu$ is
$$\mu=\frac{\alpha x^2}{\alpha^2x^2+\beta^2}=\frac 1\alpha\left(1-\frac{\beta^2}{\alpha^2x^2+\beta^2}\right)=$$
$$\frac 1\alpha\left(1-\frac 1{1+\frac{\alpha^2}{\beta^2}x^2}\right).$$
Let us carry on with the case $$\alpha=\beta=1, \ \mu =1-\frac 1{1+x^2}.$$
Since
$$\widehat{\frac 1{1+x^2}}=\frac{\sqrt{2\pi}}2 e^{-|y|}\text{ and }\widehat 1=\sqrt{2\pi}\delta_0,$$
the convolution equation from Theorem \ref{tTT} becomes
\begin{equation}\psi -\frac 12(e^{-x}F(x)+e^{x}G(x))=\frac 1{\sqrt{2\pi}},\label{e07}\end{equation}
where
\begin{equation}F(x)=\int_{-t}^x\psi(y) e^ydy\text{ and }G(x)=\int_{x}^{t}\psi(y)e^{-y}dy.\label{e08}\end{equation}
Note that since $\psi$ is even, $G(x)=F(-x)$.
Next notice that
\begin{equation}F'(x)=\psi e^x\text{ and }G'(x)=-\psi e^{-x}.\label{e09}\end{equation}
In particular,
\begin{equation}F(-t)=0,\ \psi(0)=F'(0)\text{ and from \eqref{e07} }F'(0)- F(0)=\frac 1{\sqrt{2\pi}}.\label{e010}\end{equation}
Multiplying \eqref{e07} by $e^{-x}$ we get
$$F' e^{-2x} -\frac 12 (e^{-2x} F+G)=\frac 1{\sqrt{2\pi}}e^{-x}.$$
Differentiating the last equation
we get
$$F''e^{-2x}-2F' e^{-2x}-\frac 12(-2e^{-2x}F+e^{-2x}F'+G')=-\frac 1{\sqrt{2\pi}}e^{-x}.$$
From \eqref{e09},
$$F''e^{-2x}-2F' e^{-2x}-\frac 12(-2e^{-2x}F+e^{-2x}F'-e^{-2x}F')=-\frac 1{\sqrt{2\pi}}e^{-x},
$$ and
$$F''-2F'+ F=-\frac 1{\sqrt{2\pi}}e^{x}.
$$
A particular solution is $-\frac 1{2\sqrt{2\pi}}x^2e^{x} $ and the homogeneous solution is
$Ce^x+Dxe^x$. Hence
$$F=Ce^x+Dxe^x-\frac 1{2\sqrt{2\pi}}x^2e^{x}.$$
From \eqref{e010},
$$e^tF(-t)=C-Dt-\frac {t^2}{2\sqrt{2\pi}}=0,\ \text{ and }F'(0)-F(0)=D=\frac {1}{\sqrt{2\pi}},
$$
which gives
$$C=\frac {2t+t^2}{2\sqrt{2\pi}}
$$
and
$$F=\frac {1}{2\sqrt{2\pi}}\left((2t+t^2)e^x+2xe^x-x^2e^x \right).
$$
Next, using \eqref{e09},
$$F'=\frac {1}{2\sqrt{2\pi}}\left((2t+t^2)e^x+2e^x+2xe^x-x^2e^x -2xe^x\right)$$
and
$$
\psi=\frac {1}{2\sqrt{2\pi}}\left((2t+t^2)+2+2x-x^2 -2x\right)=\frac {1}{2\sqrt{2\pi}}\left((2t+t^2)+2-x^2 \right).
$$
Then
$$\int_{-t}^t\psi_t(x)dx=\frac {1}{2\sqrt{2\pi}}\left(4t^2+2t^3 +4t -  \frac 23 t^3\right)=
\frac {1}{2\sqrt{2\pi}}\left( \frac {4}3 t^3 +4t^2 +4t          \right).
$$
Finally,
$$h^\mu=\frac d{dt}\int_{-t}^t\psi_t(x)dx=\frac {1}{2\sqrt{2\pi}}\left( 4 t^2 +8t +4          \right)
=\sqrt{\frac 2\pi}\left(  t+1         \right)^2
.$$

\end{document}